\documentclass[english,letter paper,12pt,leqno]{article}
\usepackage{stmaryrd}
\usepackage{amsmath, amscd, amssymb, mathrsfs, accents, amsfonts,amsthm}
\usepackage[all]{xy}
\usepackage{dsfont}
\usepackage{tikz}

\definecolor{Myblue}{rgb}{0,0,0.6}
\usepackage[a4paper,colorlinks,citecolor=Myblue,linkcolor=Myblue,urlcolor=Myblue,pdfpagemode=None]{hyperref}

\SelectTips{cm}{}

\newtheorem{theorem}{Theorem}[section]
\newtheorem{proposition}[theorem]{Proposition}
\newtheorem{lemma}[theorem]{Lemma}
\newtheorem{corollary}[theorem]{Corollary}
\newtheorem{setup}[theorem]{Setup}

\newtheoremstyle{example}{\topsep}{\topsep}
	{}
	{}
	{\bfseries}
	{.}
	{2pt}
	{\thmname{#1}\thmnumber{ #2}\thmnote{ #3}}
	
	\theoremstyle{example}
	\newtheorem{definition}[theorem]{Definition}
	\newtheorem{example}[theorem]{Example}
	\newtheorem{remark}[theorem]{Remark}
	\newtheorem{strat}[theorem]{Strategy}

\numberwithin{equation}{section}

\def\res{\operatorname{Res}}

\def\im{\operatorname{Im}}

\def\LG{\mathcal{LG}}

\def\Hom{\operatorname{Hom}}

\def\be{\begin{equation}}
\def\ee{\end{equation}}

\def\nZ{\mathds{Z}}
\def\nQ{\mathds{Q}}

\def\L{\mathcal{C}}
\def\ferm{\gamma}
\def\fermc{\gamma^\dagger}

\DeclareMathOperator{\End}{End}

\DeclareMathOperator{\hmf}{hmf}
\DeclareMathOperator{\HMF}{HMF}

\DeclareMathOperator{\At}{At}
\DeclareMathOperator{\Cat}{Cat}
\DeclareMathOperator{\Spec}{Spec}

\begin{document}

\def\Res{\res\!}
\newcommand{\ud}{\mathrm{d}}
\newcommand{\Ress}[1]{\res_{#1}\!}
\newcommand{\cat}[1]{\mathcal{#1}}
\newcommand{\lto}{\longrightarrow}
\newcommand{\xlto}[1]{\stackrel{#1}\lto}
\newcommand{\mf}[1]{\mathfrak{#1}}
\newcommand{\md}[1]{\mathscr{#1}}
\def\sus{\l}
\def\l{\,|\,}
\def\sgn{\textup{sgn}}

\title{The cut operation on matrix factorisations}
\author{Daniel Murfet}


\maketitle

\begin{abstract}
The bicategory $\LG$ of Landau-Ginzburg models has polynomials as objects and matrix factorisations as $1$-morphisms. The composition of these $1$-morphisms produces infinite rank matrix factorisations, which is a nuisance. In this paper we define a bicategory $\L$ which is equivalent to $\LG$ in which composition of $1$-morphisms produces finite rank matrix factorisations equipped with the action of a Clifford algebra. This amounts to a finite model of composition in $\LG$.
\end{abstract}

\tableofcontents

\section{Introduction}

The bicategory $\LG_k$ of Landau-Ginzburg models \cite{lgdual} over a noetherian $\mathbb{Q}$-algebra $k$ is a pivotal bicategory defined in terms of the geometry of isolated hypersurface singularities. This bicategory has various applications \cite{genorb,ade, kr0401268} but an obstacle to further development of the theory is that, while the objects, $1$-morphisms and $2$-morphisms of $\LG_k$ are simply described, the composition operation for $1$-morphisms has an ``infinitary'' character which makes it hard to work with. The current paper aims to resolve this problem. 

The objects of $\LG_k$ are a class of polynomials $W(x)$ called potentials, and $1$-morphisms $W(x) \lto V(y)$ are finite-rank matrix factorisations \cite{eisenbud} of the difference $V(y) - W(x)$ over $k[x,y]$. The idempotent completion of the homotopy category of such matrix factorisations is denoted $\LG_k(W,V)$. For a third potential $U(z)$, composition in $\LG_k$ is a functor
\begin{gather*}
\LG_k(V,U) \times \LG_k(W,V) \lto \LG_k(W,U)\,,\\
(Y , X) \longmapsto Y \circ X
\end{gather*}
defined for matrix factorisations $Y$ of $V - U$ and $X$ of $W - V$ by
\be
Y \circ X = \big( Y \otimes_{k[y]} X, d_Y \otimes 1 + 1 \otimes d_X \big)\,.
\ee
This is the analogue of convolution of Fourier-Mukai kernels in the matrix factorisation setting and suffers from the same defect, namely, the resulting matrix factorisation $Y \circ X$ is a free module of \emph{infinite rank} over $k[x,z]$. The problem of describing this \emph{a priori} infinite object in some finite way was studied in \cite{dm1102.2957} and the partial solution provided there was to explicitly describe a finite rank matrix factorisation together with an idempotent, splitting in the homotopy category of matrix factorisations to $Y \circ X$. However this solution is not sufficiently functorial to extent to a coherent description of the entire bicategory $\LG_k$.

The current paper refines \cite{dm1102.2957} by presenting the idempotents developed there as arising from representations of Clifford algebras (see Remark \ref{remark:relation_to_toby_paper}). This leads us to define a new bicategory $\L$ in which the objects are again potentials and the $1$-morphisms $W \lto V$ are finite rank matrix factorisations of $V - W$ additionally equipped with the structure of a representation of a Clifford algebra. This bicategory is equivalent to $\LG_k$, but it is simpler: the composition operation is described by a finite number of polynomial functions of the input data. More precisely, the composition of $1$-morphisms in $\L$ is a functor
\begin{gather}
\L(V,U) \times \L(W,V) \lto \L(W,U)\,,\label{eq:cut_functor_intro}\\
(Y,X) \longmapsto Y \l X
\end{gather}
which we call the \emph{cut operation}, and this operation is polynomial in the sense that the coefficients in the polynomials making up the matrix factorisation $Y \l X$ and its Clifford action are themselves polynomial functions of the coefficients defining $X,Y$.

 Let us describe the cut in the most important case, where $Y, X$ are matrix factorisations as above, not equipped with any additional structure as representations of a Clifford algebra. The Jacobi algebra $J_V = k[y_1,\ldots,y_m]/(\partial_{y_1} V, \ldots, \partial_{y_m} V)$ is a finitely generated free $k$-module by the assumption that $V(y)$ is a potential, and the cut is defined to be the following finite rank matrix factorisation over $k[x,z]$
\be
Y \l X = Y \otimes_{k[y]} J_V \otimes_{k[y]} X
\ee
equipped with a family of odd closed $k[x,z]$-linear operators $\{\ferm_i, \fermc_i\}_{i=1}^m$ defined by
\begin{equation}\label{eq:intro_clifford_act1_intro}
\ferm_i = \At_i\,, \qquad \fermc_i = - \partial_{y_i}(d_X) - \frac{1}{2} \sum_q \partial_{y_q} \partial_{y_i}(V) \At_{q}
\end{equation}
where the $\At_q$ are Atiyah classes (Definition \ref{defn:atiyah}). These operators satisfy Clifford relations
\be\label{eq:clifford_relations_intro}
\ferm_i \ferm_j + \ferm_j \ferm_i = 0, \qquad \fermc_i \fermc_j + \fermc_j \fermc_i = 0, \qquad \ferm_i \fermc_j + \fermc_j \ferm_i = \delta_{ij}
\ee
up to homotopy (see Theorem \ref{theorem:cut_is_rep}). The differential on $Y \l X$ and the Clifford operators can be written explicitly as matrices over $k[x,z]$, where the coefficients of the monomials in each row and column are polynomials in the monomial coefficients of $d_Y, d_X$ and $V$. Thus, the cut operation is a polynomial function of its inputs and, in particular, infinite rank matrix factorisations do not appear in the definition of $\L$.

Let us sketch where the cut operation comes from, why Clifford algebras are involved, and why there is an equivalence $\L \cong \LG_k$. In Section \ref{section:intro_appli} we describe various applications. By the results of \cite{dm1102.2957} there is an isomorphism in the homotopy category of (infinite rank) matrix factorisations of $U - W$ over $k[x,z]$ of the form
\begin{equation}\label{eq:final_finite_model_intro}
\xymatrix@C+4pc{
Y \l X \ar@<-1ex>[r]_-{\Phi^{-1}} & S_m \otimes_{k} ( Y \otimes X )\ar@<-1ex>[l]_-{\Phi}
}
\end{equation}
where $S_m = \bigwedge ( k^{\oplus m}[1] )$ denotes the exterior algebra as a $\nZ_2$-graded module, on the space $k^{\oplus m}$ placed in degree one. The main theorem (Theorem \ref{theorem:htpy_equivalence_main}) identifies the natural action of the Clifford algebra $C_m = \End_{k}(S_m)$ on the right hand side of \eqref{eq:final_finite_model_intro} with the action of $C_m$ via the aforementioned operators $\ferm_i, \fermc_i$ on the left hand side.

There is a category called the \emph{Clifford thickening} of the homotopy category of matrix factorisations of $U - W$, in which objects are matrix factorisations equipped with a $C_l$-action for some $l$ (i.e. $l$ is allowed to be different for different objects). In this category $S_m \otimes_{k} ( Y \otimes X )$ with its $C_m$-action is isomorphic to $Y \otimes X$ with no Clifford action, and so \eqref{eq:final_finite_model_intro} can be read as a natural isomorphism $Y \l X \cong Y \otimes X$. We prove this is suitably functorial in $Y,X$ and therefore defines an equivalence $\L \cong \LG_k$.

\vspace{0.5cm}

The paper is structured as follows: in Section \ref{section:background} we recall the definition of superbicategories and Clifford algebras, and introduce the Clifford thickening of a supercategory. The reason we need the formalism of supercategories is that we make extensive use of Clifford actions on objects of linear categories, and for this to make sense the categories themselves must be $\nZ_2$-graded (that is, supercategories). In Section \ref{section:lg_cut_system} we define $\L$ and in Section \ref{section:equivllg} we prove it is equivalent to $\LG_k$ as a superbicategory without units. 

\subsection{Applications}\label{section:intro_appli}

\subsubsection{$A_\infty$-algebras and minimal models}

As a special case of the deformation retract \eqref{eq:final_finite_model_intro} we obtain in Section \ref{example:computing_homs} for any potential $V(y)$ and matrix factorisations $X,X'$ of $V$ a deformation retract of $\nZ_2$-graded complexes of $k$-modules
\begin{equation}\label{eq:final_finite_model_hom}
\xymatrix@C+4pc{
(X')^{\vee} \l X \ar@<-1ex>[r]_-{\Phi^{-1}} & S_m \otimes_{k} \Hom_{k[y]}(X', X) \ar@<-1ex>[l]_-{\Phi}
}
\end{equation}
where $\Hom_{k[y]}(X',X)$ denotes the Hom-complex. Observe that if $k$ is a field this is infinite-dimensional (with finite-dimensional cohomology) while $(X')^{\vee} \l X$ is finite-dimensional.

In the case $X = X'$ the right hand side is a DG-algebra and the deformation retract may be taken as the input to the standard algorithms for producing an $A_\infty$-minimal model of $\End_{k[y]}(X)$. It turns out that this input is well-suited for actually doing computations, which we hope to return to elsewhere.

\subsubsection{A geometric enrichment of $\LG_k$}

One of our motivations for introducing $\cat{C}$ is that it has an enrichment over affine schemes. This is not true of $\LG_k$, precisely because the composition is infinitary. We hope to return to this elsewhere, but for context we include a sketch assuming $k$ is an algebraically closed field of characteristic zero. Given potentials $W(x), V(y)$, a matrix factorisation of $V(y) - W(x)$ is after choosing a homogeneous basis just a matrix
\be\label{eq:intro_capitalD}
D \in M_{2r}( k[x,y] ) \qquad \text{ satisfying } \qquad D^2 = (V(y) - W(x)) \cdot I_{2r}
\ee
where, dividing the matrix $D$ into $r \times r$ blocks, the upper left and bottom right blocks are zero matrices. If we fix the size $r$ of the blocks and bound the degrees of the monomials appearing in $D$ by an integer $s$, then we can easily parametrise matrix factorisations by the $k$-points of an affine scheme.

We may add further coordinates and equations to encode closed odd operators satisfying Clifford relations \eqref{eq:clifford_relations_intro} and, in this way, for a choice of parameters $\lambda = (r, s, \ldots)$ define a scheme $\mathscr{M}_{\lambda}(W,V)$ such that as sets
\begin{equation}\label{eq:param_frmf}
\operatorname{ob}(\cat{C}(W,V)) = \bigcup_{\lambda} \mathscr{M}_{\lambda}(W,V)_{k}
\end{equation}
where $Z_{k}$ denotes the $k$-points of a scheme $Z$. The cut operation is a polynomial function of its inputs, so there is an indexed family of morphisms of schemes
\be\label{eq:enrichment_morphisms_intro}
\mathscr{M}_\mu(V,U) \times \mathscr{M}_\lambda(W,V) \lto \mathscr{M}_{f(\lambda, \mu)}(W, U)
\ee
lifting the functor \eqref{eq:cut_functor_intro} on the level of objects for some function $f(\lambda, \mu)$ of indices. Finally, using the approach of Section \ref{example:computing_homs}, we can define for each pair of potentials $W,V$ and tuple of parameters $\lambda$ a vector bundle of Clifford representations
\be\label{eq:enrichment_intro_bundles}
\xymatrix{
\mathscr{H}\! om_{W,V,\lambda} \ar[d]\\
\mathscr{M}_\lambda(W,V) \times \mathscr{M}_\lambda(W,V)\,,
}
\ee
such that for matrix factorisations $X,X'$ of $V - W$ the fiber of this bundle over the point $(X',X)$ is the Clifford representation $(X')^{\vee} \l X$ of \eqref{eq:final_finite_model_hom}, which is isomorphic to the Hom-complex $\Hom_{k[x,y]}(X',X)$ tensored with a suitable spinor representation. 

The enrichment of $\cat{C}$ then consists of the family of schemes $\mathscr{M}_\lambda(W,V)$, the composition morphisms \eqref{eq:enrichment_morphisms_intro}, the bundles \eqref{eq:enrichment_intro_bundles}, and various other data satisfying natural constraints which encode the structure of the bicategory.

\subsubsection{Semantics of linear logic}

In forthcoming work we use the above enrichment of $\L$ to define a semantics of linear logic \cite{girard_llogic} in which formulas are interpreted by tuples $(\md{X}, k[x], W)$ where $\md{X}$ is a scheme and $W \in k[x]$ is a potential. Proofs in the logic are interpreted by certain correspondences between these pairs. Let us briefly explain the general point which leads us to use $\L$ rather than $\LG$, for which we will use the denotation of the Church numeral $\underline{2}$ (see \cite{murfet_ll}).

Let $\alpha$ be a variable of the logic and suppose it has denotation $\llbracket \alpha \rrbracket = (\Spec(k), k[x], W)$ in the semantics. Let $\mathscr{M} = \mathscr{M}(W,W)$ be the (indexed) affine scheme parametrising $1$-morphisms $X: W \lto W$ in $\L$, as discussed above. Then the semantics works as follows:
\begin{gather*}
\llbracket \alpha \multimap \alpha \rrbracket = \big(\Spec(k), k[x,x'], W(x) - W(x') \big)\,,\\
\llbracket {!}( \alpha \multimap \alpha ) \rrbracket = \big( \mathscr{M}, k, 0 \big)\,.
\end{gather*}
A proof of the sequent ${!}(\alpha \multimap \alpha) \vdash \alpha \multimap \alpha$ has for its denotation a correspondence between $\llbracket {!}( \alpha \multimap \alpha ) \rrbracket$ and $\llbracket \alpha \multimap \alpha \rrbracket$, which is a matrix factorisation of $W(x) - W(x')$ over the scheme $\mathscr{M} \times \Spec(k[x,x'])$. For example, the Church numeral $\underline{2}$ has for its denotation the matrix factorisation whose fiber over a loop $X: W \lto W$, viewed as a point of $\mathscr{M}$, is the square $X \l X$ as a matrix factorisation of $W(x) - W(x')$ over $\Spec(k[x,x'])$. For these definitions to work, the coefficients of the monomials in the matrix factorisation $X \l X$ need to be given explicitly as polynomial functions of the coefficients in $X$. This means we have to use $\L$ rather than $\LG_k$.

\medskip

\emph{Acknowledgements.} Thanks to Nils Carqueville for helpful comments on the draft.

\section{Background}\label{section:background}

Throughout $k$ is a noetherian $\mathbb{Q}$-algebra.\footnote{The noetherian hypothesis is only needed in Section \ref{section:superbicatLG} when we pass from a ring to its adic completion with respect to the Jacobi ideal, and we wish to know the morphism from the ring to its completion is flat. This completion is only done to ensure the existence of a suitable connection - if a connection can be provided by other means (e.g. a grading) then the noetherian hypothesis could be dropped.} By default categories and functors are $k$-linear. We write $\partial_x$ for the differential operator $\frac{\partial}{\partial x}$.

\subsection{Supercategories}

A superbicategory has for every $1$-morphism $F$ an associated $1$-morphism $\Psi F$ with $\Psi^2 \cong 1$, and in this sense the categories of $1$-morphisms are $\mathbb{Z}_2$-graded. The bicategory of Landau-Ginzburg models is an example of a superbicategory. Since Clifford algebras (which are $\mathbb{Z}_2$-graded) play a fundamental role, superbicategories are the natural language. Our main references for the foundations are \cite{ellis_lauda,kang,kang2}, but we give the full definitions below as these references work only with strict bicategories.

\begin{definition} A \emph{supercategory} is a category $\cat{C}$ together with a functor $\Psi: \cat{C} \lto \cat{C}$ and a natural isomorphism $\xi: \Psi^2 \lto 1_{\cat{C}}$ satisfying the condition
\[
\xi * 1_{\Psi} = 1_{\Psi} * \xi
\]
as natural transformations $\Psi^3 \lto \Psi$. A \emph{superfunctor} $(F, \gamma)$ from a supercategory $(\cat{C}, \Psi_{\cat{C}})$ to a supercategory $(\cat{D}, \Psi_{\cat{D}})$ is a functor $F: \cat{C} \lto \cat{D}$ together with a natural isomorphism $\gamma: F \Psi_{\cat{C}} \lto \Psi_{\cat{D}} F$ satisfying
\[
1_F * \xi = (\xi * 1_F ) ( 1_{\Psi} * \gamma ) ( \gamma * 1_{\Psi} )\,.
\]
A \emph{supernatural transformation} $\varphi: F \lto G$ between superfunctors $(F,\gamma_F), (G,\gamma_G)$ is a natural transformation making the following diagram commute:
\[
\xymatrix{
F \Psi \ar[r]^-{\varphi \Psi}\ar[d]_-{\gamma_F} & G \Psi \ar[d]^-{\gamma_G}\\
\Psi F \ar[r]_-{\Psi \varphi} & \Psi G
}\,.
\] 
\end{definition}

\begin{example}\label{example:Aassup} If $A$ is a $\mathbb{Z}_2$-graded $k$-algebra there is a supercategory $\cat{A}$ with set of objects $|\cat{A}| = \mathbb{Z}_2$ and morphisms $\cat{A}(0,0) = \cat{A}(1,1) = A_0, \cat{A}(0,1) = \cat{A}(1,0) = A_1$ with the obvious composition and addition rules. The functor $\Psi$ is defined on objects by $\Psi(i) = i+1$ and on morphisms by the identity, while $\xi = 1$.
\end{example}

Let $\cat{C}$ be a supercategory. For objects $X,Y \in \cat{C}$ we define the $\mathbb{Z}_2$-graded $k$-module
\[
\cat{C}^*(X,Y) = \prod_{i \in \mathbb{Z}_2} \cat{C}(X, \Psi^i Y)\,.
\]
There is for any triple $X,Y,Z$ of objects a morphism of graded modules
\[
\cat{C}^*(Y,Z) \otimes_k \cat{C}^*(X, Y) \lto \cat{C}^*(X,Z)
\]
given for instance by $\varphi \otimes \psi \mapsto \xi \Psi( \varphi ) \psi$ when $\varphi, \psi$ are both of degree one. This makes $\cat{C}$ into a category enriched over the monoidal category of $\mathbb{Z}_2$-graded $k$-modules.

There are isomorphisms of graded modules $\cat{C}^*(X, \Psi Y) \cong \Psi \cat{C}^*(X, Y) \cong \cat{C}^*(\Psi X, Y)$ where $\Psi V = V [1]$ denotes the grading shift on a $\mathbb{Z}_2$-graded $k$-module. For $\mathbb{Z}_2$-graded $k$-modules $V,W$ we write $\Hom^0_k(V, W)$ for the module of degree zero maps and $\Hom_k^*(V, W)$ for the $\mathbb{Z}_2$-graded module of homogeneous maps. If $A$ is a $\mathbb{Z}_2$-graded $k$-algebra then $A^{\operatorname{op}}$ denotes the same underlying graded module with multiplication $f * g = (-1)^{|f||g|} gf$.

\begin{definition}\label{defn:algebra_modules} Let $A$ be a $\mathbb{Z}_2$-graded $k$-algebra and $\cat{C}$ a supercategory. A \emph{left $A$-module in $\cat{C}$} is an object $X \in \cat{C}$ together with a morphism of $\mathbb{Z}_2$-graded algebras $A \lto \cat{C}^*(X,X)$. A \emph{morphism} of left $A$-modules is a morphism in $\cat{C}$ which commutes with the $A$-action in the obvious sense. The category of left $A$-modules in $\cat{C}$ is denoted $\cat{C}_A$. Right modules are defined similarly, using $A^{\operatorname{op}}$.
\end{definition}

\begin{remark}\label{remark:supercat_idempcomp} Viewing $A$ as a supercategory $\cat{A}$ as in Example \ref{example:Aassup}, there is an equivalence between the category of $A$-modules in $\cat{C}$ and the category of superfunctors $\cat{A} \lto \cat{C}$ and supernatural transformations.
\end{remark}

\begin{remark}\label{remark:idempotent_completion} The \emph{idempotent completion} $\cat{C}^\omega$ of a category $\cat{C}$ has as objects pairs $(X,e)$ consisting of an object $X \in \cat{C}$ and an idempotent endomorphism $e: X \lto X$. In $\cat{C}^\omega$ a morphism $\phi: (X,e) \lto (Y,e)$ is a morphism $\phi: X \lto Y$ satisfying $\phi e = \phi$ and $e \phi = \phi$. The identity on $(X,e)$ is $e$. There is a fully faithful functor $\iota: \cat{C} \lto \cat{C}^{\omega}, \iota(X) = (X,1_X)$ which is an equivalence if $\cat{C}$ is idempotent complete (i.e. all idempotents split in $\cat{C}$).

If $\cat{C}$ is a supercategory then there is a functor $\Psi^\omega: \cat{C}^\omega \lto \cat{C}^\omega$ defined on objects by $\Psi(X,e) = (\Psi X, \Psi e)$ and a natural isomorphism $\xi: \Psi^\omega \circ \Psi^\omega \lto 1_{\cat{C}^\omega}$ defined by $\xi_{(X,e)} = e \xi_X$. The tuple $(\cat{C}^\omega, \Psi^\omega, \xi^\omega)$ is a supercategory.
\end{remark}

For background on bicategories see \cite{bor94, benabou, gray, kellystreet, lack}. We follow the notation of \cite{lgdual}, so that lower case letters $a,b, \ldots$ denote objects of a bicategory, while upper case letters $X,Y, \ldots$ and greek letters $\alpha, \beta, \ldots$ respectively denote $1$-morphisms and $2$-morphisms. Units are denoted $\Delta$, associators are $\alpha$, and unitors are $\lambda, \rho$. The composition of $Y,X$ is denoted $YX$ or $Y \circ X$.

\begin{definition} A \emph{superbicategory} is a bicategory $\cat{B}$ together with the data:
\begin{itemize}
\item For each object $a$ a $1$-morphism $\Psi_a: a \lto a$ and a $2$-isomorphism $\xi_a: \Psi^2_a \lto \Delta_a$.
\item For each $1$-morphism $X: a \lto b$ a natural $2$-isomorphism $\gamma_X: X \Psi_a \lto \Psi_b X$.
\end{itemize}
This data is required to satisfy the following axioms:
\begin{itemize}
\item For each composable pair $X,Y$ of $1$-morphisms the diagram
\[
\xymatrix@C+2pc@R+1pc{
(YX) \Psi \ar[rrr]^-{\gamma_{YX}} \ar[d]_-{\alpha} & & & \Psi ( YX )\\
Y ( X \Psi ) \ar[r]_-{1_Y * \gamma_X} & Y( \Psi X ) \ar[r]_-{\alpha^{-1}} & ( Y \Psi ) X \ar[r]_-{\gamma_Y * 1_X} & (\Psi Y ) X \ar[u]_-{\alpha}
}
\]
commutes.
\item For every object $a$, $\xi_a * 1_\Psi = 1_\Psi * \xi_a$.
\item For every $1$-morphism $X : a \lto b$, $1_X * \xi_a = ( \xi_b * 1_X ) ( 1_\Psi * \gamma_X ) (\gamma_X * 1_\Psi )$.
\end{itemize}
\end{definition}

\begin{example} Small supercategories, superfunctors and supernatural transformations form a superbicategory $\Cat^{\operatorname{sup}}_k$.
\end{example}

A superbicategory can be constructed out of a bicategory $\cat{B}$ in which the categories $\cat{B}(a,b)$ are all equipped with the structure of supercategories; see Appendix \ref{section:constructing_superbicategories}.

\begin{example}\label{example:bicategory_m} There is a bicategory of $\mathbb{Z}_2$-graded $k$-algebras where $1$-morphisms are $\mathbb{Z}_2$-graded bimodules and $2$-morphisms are degree zero bimodule maps. Given a $B$-$A$-module $M$ the shift $M[1]$ has the grading $M[1]_i = M_{i+1}$ and the left and right action given by
\[
b \cdot m = (-1)^{|b|} bm, \qquad m \cdot a = (-1)^{|a|} m a\,.
\]
This is a functor $\Psi = (-)[1]$ on the category of $B$-$A$-bimodules and with $\xi = 1$ this defines the structure of a supercategory on this category of bimodules. The usual isomorphisms of bimodules $\tau$
\begin{align}
N[1] \otimes M &\lto (N \otimes M)[1], \qquad n \otimes m \mapsto n \otimes m \label{eq:shift_iso_a}\\
N \otimes M[1] &\lto (N \otimes M)[1], \qquad n \otimes m \mapsto (-1)^{|n|} n \otimes m\label{eq:shift_iso_b}
\end{align}
satisfy the conditions in Appendix \ref{section:constructing_superbicategories} and therefore give the bicategory of $\mathbb{Z}_2$-graded algebras and bimodules the structure of a superbicategory.
\end{example}

\begin{definition}\label{defn:laxsuperfunctor} Given two superbicategories $\cat{B}, \cat{C}$ a \emph{lax superfunctor} $\Phi: \cat{B} \lto \cat{C}$ is a lax functor together with, for each object $a$ in $\cat{B}$, a $2$-morphism
\[
\kappa_a: \Psi_{\Phi a} \lto \Phi( \Psi_a )
\]
satisfying two coherence conditions:
\begin{itemize}
\item[(a)] for all $a$, commutativity of
\[
\xymatrix@C+2pc{
\Psi_{\Phi a} \Psi_{\Phi a} \ar[dd]_-{\xi}\ar[r]^-{\kappa * \kappa} & \Phi( \Psi_a ) \Phi( \Psi_a ) \ar[d]\\
& \Phi( \Psi_a \Psi_a ) \ar[d]^-{\Phi(\xi)}\\
\Delta_{\Phi a} \ar[r] & \Phi( \Delta_a )
}
\]
\item[(b)] for each $1$-morphism $X: a \lto b$ in $\cat{B}$, commutativity of
\[
\xymatrix@C+2pc{
\Phi(X) \Psi_{\Phi a} \ar[d]_-{\gamma} \ar[r]^-{1 * \kappa} & \Phi(X) \Phi(\Psi_a) \ar[r] & \Phi( X \Psi_a ) \ar[d]^-{\Phi(\gamma)}\\
\Psi_{\Phi b} \Phi(X) \ar[r]_-{\kappa * 1} & \Phi(\Psi_b) \Phi(X) \ar[r] & \Phi( \Psi_b X)
}
\]
\end{itemize}
A \emph{strong superfunctor} is a lax superfunctor with $\kappa_a$ an isomorphism for every object $a$.
\end{definition}

\subsection{The superbicategory of Landau-Ginzburg models}\label{section:superbicatLG}

A polynomial $W \in k[x_1,\ldots,x_n]$ is a \emph{potential} if it satisfies the three conditions set out in \cite[Section 2.2]{lgdual}, the two most important being that the partial derivatives $\partial_{x_1} W, \ldots, \partial_{x_n} W$ form a quasi-regular sequence in $k[x]$ and that the quotient $J_W = k[x]/(\partial_{x_1} W, \ldots, \partial_{x_n} W)$ is a finitely generated free $k$-module. Typical examples are the ADE singularities \cite[I \S 2.4]{greuel} of which the simplest are the $A_N$-singularities $W_{A_N} = x_1^{N+1} + x_2^2 + \cdots + x_n^2$ .

A \emph{matrix factorisation} of $W$ over $k[x]$ is a $\mathbb{Z}_2$-graded free $k[x]$-module $X = X^0 \oplus X^1$ together with an odd operator (the differential) $d_X: X \lto X$ satisfying $d_X^2 = W \cdot 1_X$. A matrix factorisation $(X,d_X)$ is \emph{finite rank} if the underlying free module is finitely generated. In this case we may, after choosing a homogeneous basis, write
\[
d_X = \begin{pmatrix} 0 & d_X^1 \\ d_X^0 & 0 \end{pmatrix}
\]
for matrices of polynomials $d_X^0, d_X^1$. Morphisms of matrix factorisations are degree zero $k[x]$-linear maps which commute with the differentials. There is a homotopy relation on morphisms, and the \emph{homotopy category} $\HMF(k[x],W)$ of matrix factorisations has as objects matrix factorisations and as morphisms the equivalence classes of morphisms of matrix factorisations up to homotopy. By $\hmf(k[x],W)$ we denote the full subcategory of finite rank matrix factorisations. For background we refer to \cite{yoshino98}.

By $\hmf(k[x],W)^{\oplus}$ we denote the full subcategory of $\HMF(k[x],W)$ consisting of matrix factorisations which are direct summands (in the homotopy category) of finite rank matrix factorisations. Since $\HMF(k[x],W)$ is idempotent complete, there is an equivalence (see Remark \ref{remark:idempotent_completion} for the notation)
\be\label{eq:omegavsoplus}
\hmf(k[x],W)^\omega \lto \hmf(k[x],W)^{\oplus}
\ee
sending a pair $(X,e)$ to an infinite rank matrix factorisation splitting the idempotent $e$.

If $(X,d_X)$ is a matrix factorisation then so is $(X[1], -d_X)$, and this defines a functor $\Psi = (-)[1]$ on the homotopy category of matrix factorisations. Together with the identity $\xi: \Psi^2 \lto 1$ this makes $\HMF(k[x],W)$ and $\hmf(k[x],W)$ into supercategories.

The bicategory of Landau-Ginzburg models $\LG = \LG_k$ has for its objects pairs $(x,W)$ consisting of an ordered set of variables $x = (x_1,\ldots,x_n)$ and a potential $W \in k[x] = k[x_1,\ldots,x_n]$. Given potentials $W \in k[x]$ and $V \in k[y]$ the supercategory $\cat{LG}(W,V)$ is
\[
\cat{LG}(W,V) = \hmf(k[x,y], V - W)^{\oplus}\,.
\]
Given a third potential $U \in k[z]$ the composition law is a functor
\[
\LG(W,V) \otimes_k \LG(V, U) \lto \LG(W, U)
\]
which sends a pair of $1$-morphisms $X: W \lto V$ and $Y: V \lto U$ to the tensor product
\begin{equation}\label{eq:tensor_comp}
Y \otimes X = ( Y \otimes_{k[y]} X, d_{Y \otimes X} = d_Y \otimes 1 + 1 \otimes d_X )\,.
\end{equation}
Thus $Y \circ X := Y \otimes X$. This statement requires some care: $k[x,y,z]$ is an infinite rank free module over $k[x,z]$ and so $Y \otimes X$ is an infinite rank matrix factorisation of $U - W$ over $k[x,z]$. However one can prove that $Y \otimes X$ is a direct summand of a finite rank matrix factorisation \cite{dm1102.2957} and therefore a valid object in $\LG(W,U)$.

Let $W \in k[x_1,\ldots,x_n]$ be a potential and $W(x')$ the same polynomial but in a second set of variables $x_1',\ldots,x_n'$. Using formal symbols $\Theta_i$ of odd degree, we define the $k[x,x']$-module
\[
\Delta_W = \bigwedge \big( \bigoplus_{i=1}^n k[x,x'] \cdot \Theta_i \big)\,.
\]
Equipped with a certain differential (see \cite[(2.15)]{lgdual}) this is a matrix factorisation of $W(x) - W(x')$ and defines the unit in $\cat{LG}(W,W)$ for composition of $1$-morphisms. For the associator and unitor maps $\rho: X \otimes \Delta_W \lto X$ and $\lambda: \Delta_V \otimes X \lto X$ we also refer to \emph{ibid.}

\begin{lemma} $\LG$ is a superbicategory.
\end{lemma}
\begin{proof}
The categories $\LG(W,V)$ are all naturally supercategories, and the same isomorphisms as in \eqref{eq:shift_iso_a}, \eqref{eq:shift_iso_b} define the necessary natural isomorphisms $\tau$ to equip $\LG$ with the structure of a superbicategory with $\Psi_W = \Delta_W[1]$, using Appendix \ref{section:constructing_superbicategories}.
\end{proof}

The superbicategory structure on $\LG$ is used implicitly in \cite[Section 7]{lgdual}.

\subsection{Partial derivatives as homotopies}\label{section:partial}

Let $R = k[x_1,\ldots,x_n]$ be a polynomial ring and $X$ a finite rank matrix factorisation of $W$ over $R$. Let us fix a homogeneous basis for $X$ so that we may identify $d_X$ with a matrix, and differentiate this matrix entry-wise to obtain the matrix $\lambda_i = \partial_{x_i}(d_X)$. It is easy to check that

\begin{lemma}\label{lemma:homotopy_indept} The odd $R$-linear operator $\lambda_i$ on $X$ is independent, up to homotopy, of the choice of basis.
\end{lemma}

\begin{lemma} As operators on $X$ we have $[ d_X, \lambda_i ] = \partial_{x_i} W \cdot 1_X$.
\end{lemma}
\begin{proof}
Applying the Leibniz rule to the equation $d_X d_X = W$ we obtain
\be\label{eq:partials_are_homotopies}
d_X \partial_{x_i}(d_X) + \partial_{x_i}(d_X) d_X = \partial_{x_i} W
\ee
as claimed.
\end{proof}

It follows that any element in the Jacobi ideal $( \partial_{x_1}W , \ldots, \partial_{x_n} W )$ acts null-homotopically on any morphism in the homotopy category $\hmf(R,W)$. Moreover these homotopies $\lambda_i$ are central, in the following sense:

\begin{lemma}\label{lemma:naturalityoflambda} Let $\phi: X \lto X'$ be a morphism of matrix factorisations and write $\lambda_i = \partial_{x_i}(d_X)$ and $\lambda'_i = \partial_{x_i}(d_{X'})$. There is a homotopy $\phi \lambda_i \simeq \lambda'_i \phi$.
\end{lemma}
\begin{proof}
From $d_{X'} \phi = \phi d_X$ we deduce
\[
\partial_{x_i}( d_{X'} ) \phi + d_{X'} \partial_{x_i}( \phi ) = \partial_{x_i}( \phi ) d_X + \phi \partial_{x_i}( d_X )
\]
and hence
\[
\partial_{x_i}( d_{X'} ) \phi - \phi \partial_{x_i}( d_X ) = \partial_{x_i}(\phi) d_X - d_{X'} \partial_{x_i}(\phi)
\]
as claimed.
\end{proof}

\begin{lemma}\label{lemma:commutator_of_lambdas} There is a homotopy $[ \lambda_i, \lambda_j ] \simeq \partial_{x_i} \partial_{x_j}(W) \cdot 1_X$.
\end{lemma}
\begin{proof}
Differentiating \eqref{eq:partials_are_homotopies} with respect to $j$ leaves
\[
\partial_{x_j}(d_X) \partial_{x_i}(d_X) + d_X \partial_{x_j} \partial_{x_i}(d_X) + \partial_{x_j} \partial_{x_i}(d_X) d_X + \partial_{x_i}(d_X) \partial_{x_j}(d_X) = \partial_{x_i} \partial_{x_j}W
\]
from which we read off
\[
\lambda_i \lambda_j + \lambda_j \lambda_i = \partial_{x_i} \partial_{x_j}W + [d, -\partial_{x_j} \partial_{x_i}(d_X)]\,.
\]
\end{proof}

\subsection{Clifford algebras}\label{section:clifford_algs}

The cut operation in $\L$ produces objects which are representations of Clifford algebras. We briefly recall the basic theory of these algebras and their modules; for a full discussion see \cite{friedrich}. For $n \ge 0$ the Clifford algebra $C_n$ is the associative $\mathbb{Z}_2$-graded $k$-algebra generated by odd elements $\ferm_1,\ldots,\ferm_n$ and $\fermc_1, \ldots, \fermc_n$ subject to the following relations
\be\label{eq:clifford_relations}
[\ferm_i, \ferm_j] = 0, \qquad [\fermc_i, \fermc_j] = 0, \qquad [\ferm_i, \fermc_j] = \delta_{ij}
\ee
for all $1 \le i, j \le n$. In this paper all our commutators are graded, $[a,b] = ab - (-1)^{|a||b|} ba$. We set $C_0 = k$. These Clifford algebras have essentially only one nontrivial representation, sometimes called the \emph{spinor representation} \cite[p.14]{friedrich}, defined as follows:

\begin{definition} We use the $\nZ_2$-graded $k$-module
\be
F_n = k \theta_1 \oplus \cdots \oplus k \theta_n
\ee
where the $\theta_i$ are formal variables of odd degree, and set
\be
S_n = \bigwedge F_n = \bigwedge\big( k \theta_1 \oplus \cdots \oplus k \theta_n \big)\,.
\ee
\end{definition}

\begin{definition}\label{defn:contraction} Left multiplication in the exterior algebra defines an odd operator
\begin{gather*}
\theta_i \wedge (-): S_n \lto S_n\,,\\
\theta_{j_1} \cdots \theta_{j_r} \mapsto \theta_i \theta_{j_1} \cdots \theta_{j_r}\,.
\end{gather*}
Contraction from the left defines an odd operator
\begin{gather*}
\theta_i^*\, \lrcorner\, (-): S_n \lto S_n\,,\\
\theta_{j_1} \cdots \theta_{j_r} \mapsto \sum_{l=1}^r (-1)^{(l-1)} \delta_{i,j_l} \theta_{j_1} \cdots \widehat{ \theta_{j_l} } \cdots \theta_{j_r}\,.
\end{gather*}
\end{definition}

Usually we write simply $\theta_i$ for the operator $\theta_i \wedge (-)$ and $\theta_i^*$ for the operator $\theta_i^*\, \lrcorner\, (-)$. The spinor representation is defined by the next lemma, which also shows that the algebras $C_n$ are Morita trivial in the sense that they are isomorphic to the endomorphism algebras of a finite rank free $\mathbb{Z}_2$-graded $k$-module.

\begin{lemma}\label{defn:spinorrep} The map $C_n \lto \End_k(S_n)$ defined by
\[
\fermc_i \mapsto \theta_i \wedge (-), \qquad \ferm_i \mapsto \theta_i^*\, \lrcorner\, (-)
\]
is an isomorphism of $\mathbb{Z}_2$-graded $k$-algebras.
\end{lemma}

In particular this means that every $\mathbb{Z}_2$-graded $C_n$-module is isomorphic to $S_n \otimes_k V$ for some $\mathbb{Z}_2$-graded $k$-module $V$. Using the isomorphism of $k$-modules
\[
F_m \oplus F_n \cong F_{m+n}
\]
which maps the ordered basis $\theta_1,\ldots,\theta_m$ of $F_m$ to the first $m$ basis elements of $F_{m+n}$ in their given order, and the ordered basis of $F_n$ to the last $n$ basis elements, we define
\[
S_m \otimes_k S_n = \bigwedge F_m \otimes_k \bigwedge F_n \cong \bigwedge F_{m+n} = S_{m+n}\,.
\]
From this we deduce an isomorphism of algebras
\begin{equation}\label{eq:algebra_A_additive}
C_m \otimes_k C_n \cong C_{m+n}\,.
\end{equation}
Given $m, n \ge 0$ we introduce notation for the $C_m$-$C_n$-bimodule
\begin{equation}
S_{m,n} = S_m \otimes_k S_n^* \cong \Hom_k(S_n,S_m)
\end{equation}
where $S_n^* = \Hom_k(S_n, k)$. There is an isomorphism of $C_l$-$C_n$-bimodules
\begin{equation}\label{eq:isosbimodule}
S_{l,m} \otimes_{C_m} S_{m,n} = S_l \otimes_k S_m^* \otimes_{C_m} S_m \otimes_k S_n^* \cong S_{l,n}\,.
\end{equation}

\begin{definition}\label{defn:idempotent_e} For $n \ge 0$ let $e_n \in C_n$ denote the element
\[
e_n = \ferm_1 \cdots \ferm_n \fermc_n \cdots \fermc_1\,.
\]
This is the idempotent corresponding to the summand $k \cdot 1$ of $S_n$.
\end{definition}

\begin{definition}\label{defn:idempotent_gen} More generally, given $m, n \ge 0$ consider the $k$-linear map
\[
\xymatrix{
S_n \ar[r] & k \cdot 1 \ar[r]^{\operatorname{inc}} & S_m
}
\]
given by the projection to $k \cdot 1$ followed by the inclusion into $S_m$. This defines an element
\[
\iota_{m,n} \in \Hom_k(S_n, S_m) = S_m \otimes_k S_n^*\,.
\]
Then obviously $\iota_{n,n} = e_n$.
\end{definition}

\begin{remark} Recall that given a finite rank free $k$-module $V$ and a symmetric bilinear form $B: V \otimes_k V \lto k$ the associated Clifford algebra $\operatorname{Cl}(V, B)$ is the associative $k$-algebra generated by the elements of $V$ subject to the relations
\[
vw + wv = B(v,w) \cdot 1\,.
\]
This is a $\mathbb{Z}_2$-graded $k$-algebra when we assign $|v| = 1$ for all $v \in V$ and $|1| = 0$. If we take $F = k^n, V = F \oplus F^*$ with the bilinear form $B$ under which $B(F,F) = 0, B(F^*, F^*) = 0$ and for $x \in F$ and $\nu \in F^*$, $B(x, \nu) = \frac{1}{2} \nu(x)$, then the algebra $C_n$ defined above is the Clifford algebra of the pair $(V,B)$.
\end{remark}

\subsection{The Clifford thickening}\label{section:cliff_thick}

Let $\cat{T}$ be a small idempotent complete supercategory. We construct a new supercategory $\cat{T}^\bullet$ called the \emph{Clifford thickening} in which the objects are pairs $(X,n)$ of an integer $n \ge 0$ and a left $C_n$-module $X$ in $\cat{T}$. Recall that if $A$ is a $\mathbb{Z}_2$-graded $k$-algebra then $\cat{T}_A$ denotes the supercategory of $A$-modules in $\cat{T}$ with $A$-linear maps (see Definition \ref{defn:algebra_modules}).


\begin{definition} Let $\cat{M}$ denote the superbicategory of Morita trivial $\mathbb{Z}_2$-graded $k$-algebras. The $1$-morphisms are $\mathbb{Z}_2$-graded bimodules which are finitely generated and projective over $k$ and the $2$-morphisms are degree zero bimodule maps (see Example \ref{example:bicategory_m}).
\end{definition}

We will need to tensor objects of $\cat{T}$ with $k$-modules, which is recalled in Appendix \ref{section:tensorproduct_supcat}. 

\begin{proposition} The assignment of the supercategory $\cat{T}_A$ to an algebra $A$ and of the superfunctor $\Phi_V = V \otimes_A (-)$ to a $B$-$A$-bimodule $V$ determines a strong superfunctor
\[
\Phi^{\cat{T}}: \cat{M} \lto \Cat^{\operatorname{sup}}_k
\]
to the superbicategory of small supercategories and superfunctors.
\end{proposition}
\begin{proof}
The isomorphism $\Phi_W \circ \Phi_V \cong \Phi_{W \otimes V}$ is given by Lemma \ref{lemma:assoctensor}, and apart from that the only checking that needs to be done are straightforward coherence diagrams.
\end{proof}

Next we construct a strong functor into $\cat{M}$ which picks out the Clifford algebras.

\begin{definition}
Let $\mathbb{N}$ denote the category of integers $n \ge 0$ with a unique morphism $\phi_{m,n}: n \lto m$ for each pair $m,n$. 
\end{definition}

We view $\mathbb{N}$ as a bicategory with only identity $2$-morphisms.

\begin{lemma} There is a strong functor $\mathbb{N} \lto \cat{M}$ defined by
\begin{gather*}
n \mapsto C_n = \End_k( S_n ),\\
\phi_{m,n} \mapsto S_{m,n} = S_m \otimes_k S_n^*\,.
\end{gather*}
\end{lemma}

The composite of these strong functors is a strong functor
\begin{equation}\label{eq:strongfunctorgroth}
\xymatrix@C+1pc{
\mathbb{N} \ar[r] & \cat{M} \ar[r]^-{\Phi^{\cat{T}}} & \Cat^{\operatorname{sup}}_k
}
\end{equation}
sending $n$ to the category of left $C_n$-modules in $\cat{T}$ and $\phi_{m,n}$ to the functor $S_{m,n} \otimes_{C_n} -$. There is a close connection between strong functors into the bicategory of small categories, and what are called cofibered categories, given by the Grothendieck construction \cite{vistoli}.

\begin{definition} The \emph{Clifford thickening} $\cat{T}^\bullet$ of the supercategory $\cat{T}$ is the category which results from the Grothendieck construction applied to the strong functor \eqref{eq:strongfunctorgroth}.
\end{definition}

More concretely $\cat{T}^\bullet$ is the category with objects all pairs $(X,n)$ of an integer $n \ge 0$ and an object $X \in \cat{T}$ together with a left $C_n$-action. A morphism $\alpha: (X,n) \lto (Y,m)$ is a morphism of $C_m$-modules $\alpha: S_{m,n} \otimes_{C_n} X \lto Y$. Composition of a pair of morphisms
\begin{equation}\label{eq:comp_chain0}
\xymatrix{
(X,n) \ar[r]^-{\alpha} & (Y,m) \ar[r]^-{\beta} & (Z,l)
}
\end{equation}
is given by the morphism of $C_l$-modules
\begin{equation}\label{eq:comp_chain1}
\xymatrix@C+1pc{
S_{l, n} \otimes_{C_n} X \ar[r]^-{\cong} & (S_{l,m} \otimes_{C_m} S_{m,n}) \otimes_{C_n} X \ar[d]^-{\cong}\\
& S_{l,m} \otimes_{C_m} ( S_{m,n} \otimes_{C_n} X ) \ar[r]^-{ 1 \otimes \alpha } & S_{l,m} \otimes_{C_m} Y \ar[r]^-{\beta} & Z\,.
}
\end{equation}
There is a canonical functor $p: \cat{T}^\bullet \lto \mathbb{N}$ defined by $p(X,n) = n$, which defines a cofibered category over $\mathbb{N}$ by the definition of the Grothendieck construction. This definition has the benefit of being derived from a general construction, but there is an alternative description of morphisms which is less cumbersome:

\begin{definition} Let $\cat{T}^{\star}$ be the category with the same objects as $\cat{T}^\bullet$ and
\begin{align*}
\cat{T}^{\star}( (X,n), (Y,m) ) &= \Big\{ \alpha \in \cat{T}(X,Y) \l \alpha \fermc_i = 0 \text{ for } 1 \le i \le n \text{ and }\\
&\qquad \qquad\qquad \qquad \;\; \ferm_i \alpha = 0 \text{ for } 1 \le i \le m \Big\}\,.
\end{align*}
Composition is as in $\cat{T}$, but the identity on $(X,n)$ in $\cat{T}^\star$ is the morphism
\[
1^*_{(X,n)} = e_n
\]
where $e_n$ is the action on $X$ of the canonical idempotent from Definition \ref{defn:idempotent_e}.
\end{definition}

\begin{lemma}\label{lemma:equivcliffordthick} $\cat{T}^{\star}$ is a category and there is an equivalence of categories
\[
\Theta: \cat{T}^\bullet \lto \cat{T}^{\star}
\]
\end{lemma}
\begin{proof}
To prove that $\cat{T}^{\star}$ is a category it only needs to be checked that $1^*_{(X,n)}$ really acts as an identity. Let $\alpha: X \lto Y$ be a morphism in $\cat{T}^{\star}$ and consider
\begin{align*}
\alpha e_n &= \alpha \ferm_1 \cdots \ferm_n \fermc_n \cdots \fermc_1\\
&= \alpha \left( \ferm_1 \cdots \ferm_{n-1} \fermc_{n-1} \cdots \fermc_1 - \ferm_1 \cdots \fermc_n \ferm_n \cdots \fermc_1 \right)\\
&= \alpha \ferm_1 \cdots \ferm_{n-1} \fermc_{n-1} \cdots \fermc_1
\end{align*}
since in the second summand within the brackets, there is nothing to stop $\fermc_n$ from commuting to the left and annihilating with $\alpha$. Proceeding inductively, we see that $\alpha e_n = \alpha$. The proof that $e_n \alpha = \alpha$ is similar.

The functor $\Theta$ is defined to be the identity on objects, and given $\beta: (X,n) \lto (Y,m)$ in $\cat{T}^{\bullet}$ which is a morphism $\beta: S_{m,n} \otimes_{C_n} X \lto Y$ in $\cat{T}$ we define
\begin{gather*}
\Theta(\beta): X \lto Y\,,\\
\Theta(\beta)( x ) = \beta( \iota_{m,n} \otimes x )
\end{gather*}
where $\iota_{m,n}$ is as in Definition \ref{defn:idempotent_gen}. We relegate the proof that this functor gives a bijection on morphism sets to Lemma \ref{lemma:morphisms_two_forms}. Functoriality is easily checked, and $\Theta(1_X) = 1^*_{(X,n)}$ since by definition $\iota_{n,n} = e_n$.
\end{proof}

Every $C_n$-module $X$ in $\cat{T}$ is of the form $S_n \otimes_k V$ for some object $V$. It will be important later to know exactly how to extract the object $V$ from $X$ by splitting an idempotent.

\begin{lemma}\label{lemma:whackamole} Let $X$ be a $C_n$-module in $\cat{T}$. The idempotent $e_n: X \lto X$ splits, say 
\[
\xymatrix@C+2pc{
X \ar@<1ex>[r]^-{f_0} & V \ar@<1ex>[l]^-{g_0}
}
\]
with $f_0 g_0 = 1_V$ and $g_0 f_0 = e_n$ and in $\cat{T}^\bullet$ there are mutually inverse isomorphisms
\[
\xymatrix@C+2pc{
(X, n) \ar@<1ex>[r]^-{f} & (V,0) \ar@<1ex>[l]^-{g}
}
\]
induced by $f_0,g_0$. In particular
\be
S_n^* \otimes_{C_n} X \cong V\,, \qquad S_n \otimes_k V \cong X\,.
\ee
\end{lemma}
\begin{proof}
The map $f_0$ is a morphism $(X,n) \lto (V,0)$ in $\cat{T}^{\star}$ since
\[
f_0 \fermc_i = f_0 e_n \fermc_i = f_0 \ferm_1 \cdots \ferm_n \fermc_n \cdots \fermc_1 \fermc_i = 0\,.
\]
Similarly $g_0$ is a morphism $(V,0) \lto (X,n)$ in $\cat{T}^{\star}$. Let $f,g$ be the morphisms in $\cat{T}^\bullet$ corresponding to $f_0,g_0$ under the equivalence of Lemma \ref{lemma:equivcliffordthick}. To prove that $fg = 1_{(X,n)}$ and $gf = 1_{(V,0)}$ in $\cat{T}^{\bullet}$ it suffices to prove that $g_0 f_0 = 1^*_{(X,n)} = e_n$ and $f_0 g_0 = 1^*_{(V,0)} = 1_V$ in $\cat{T}$. But this is true by definition.
\end{proof}

\begin{lemma}\label{lemma:embedcliffordthick} $\cat{T}^\bullet$ is a supercategory and the functor $\iota: \cat{T} \lto \cat{T}^\bullet$ defined by $\iota(X) = (X,0)$ is an equivalence of supercategories.
\end{lemma}
\begin{proof}
The supercategory structure is given by the functor $\Psi: \cat{T}^\bullet \lto \cat{T}^\bullet$ where $\Psi(X,n) = ( \Psi X, n )$ and $\Psi X$ has the $C_n$-action given by Definition \ref{defn:psimodule}. The functor $\iota$ is fully faithful and it follows from Lemma \ref{lemma:whackamole} that it is essentially surjective.
\end{proof}

\section{The superbicategory $\L$}\label{section:lg_cut_system}


In this section we define a superbicategory without units $\L$. Later we will prove that this bicategory is equivalent to $\LG$, although this will not be apparent at first. The objects of $\L$ are the same as $\LG$, namely potentials $(x,W)$.

\begin{definition} Given potentials $(x,W)$ and $(y,V)$ we define
\be\label{eq:defnluv}
\L(W,V) = \Big( \hmf( k[x,y], V - W )^{\omega} \Big)^{\bullet}\,,
\ee
where $(-)^\omega$ denotes the idempotent completion (see Remark \ref{remark:idempotent_completion}) and $(-)^{\bullet}$ is the Clifford thickening (see Section \ref{section:cliff_thick}).
\end{definition}

Thus a $1$-morphism $W \lto V$ in $\L$ is a finite rank matrix factorisation of $V(y) - W(x)$ together with an idempotent endomorphism $e$ and a family of odd operators $\ferm_i, \fermc_i$ satisfying Clifford relations and the equations (all up to homotopy)
\begin{gather*}
\ferm_i e = \ferm_i = e \ferm_i\,, \qquad \fermc_i e = \fermc_i = e \fermc_i\,.
\end{gather*}
For matrix factorisations $X,Y$ of $V(y) - W(x)$ with respective Clifford actions $\{ \ferm_i, \fermc_i \}_{i=1}^a$ and $\{ \rho_j, \rho_j^\dagger \}_{j=1}^b$ the $2$-morphisms $\phi: (X,a) \lto (Y,b)$ in $\L$ are by Lemma \ref{lemma:equivcliffordthick} in bijection with morphisms of matrix factorisations $\phi: X \lto Y$ satisfying
\[
\phi \fermc_i = 0\,, \qquad \rho_j \phi = 0\,, \qquad 1 \le i \le a, \quad 1 \le j \le b\,.
\]
Observe that the homotopy category $\LG(W,V)_{fin}$ of finite rank matrix factorisations sits fully faithfully inside $\L(W,V)$ as the subcategory where all idempotents are identities and there are no Clifford actions, and by Lemma \ref{lemma:embedcliffordthick} the inclusion of this subcategory into $\L(W,V)$ is an idempotent completion.

 The first aim of this section is to define, for any object $(z,U)$ of $\LG$, a functor
\begin{gather*}
\L(V,U) \otimes_k \L(W,V) \lto \L(W, U)\,\\
(Y,X) \mapsto Y \l X
\end{gather*}
which we call the \textsl{cut functor}. The cut operation is defined on matrix factorisations $X$ of $V - W$ and $Y$ of $U - V$ as follows, assuming $y = (y_1,\ldots,y_m)$ and writing
\be\label{eq:defnjacobian}
J_V = k[y] / ( \partial_{y_1} V, \ldots, \partial_{y_m} V )\,.
\ee

\begin{lemma} The $\nZ_2$-graded $k[x,z]$-module
\[
Y \l X = Y \otimes_{k[y]} J_V \otimes_{k[y]} X
\]
with the differential $d_Y \otimes 1 + 1 \otimes d_X$ is a finite rank matrix factorisation of $U - W$.
\end{lemma}
\begin{proof}
Since $V$ is a potential $J_V$ is a finite rank free $k$-module, and it follows that $Y \l X$ is a finite rank free $k[x,z]$-module. 
\end{proof}

\begin{remark}\label{remark:inflation} If we choose a $k$-basis for $J_V$ then multiplication by $y_i$, as a $k$-linear operator $J_V \lto J_V$, gives a matrix $[y_i]$ over $k$. The cut $Y \l X$ has a differential which, as a matrix, can be described by taking the matrix of $d_Y \otimes 1 + 1 \otimes d_X$ over $k[x,y,z]$ and replacing every occurrence of $y_i$ by the scalar matrix $[y_i]$ and $1 \in k$ by the identity matrix. This ``inflation'' produces a matrix of polynomials over $k[x,z]$. 
\end{remark}

\begin{definition} Given morphisms of matrix factorisations $\phi: X \lto X'$ and $\psi: Y \lto Y'$ we define the morphism of matrix factorisations
\[
\psi \l \phi: Y \l X \lto Y' \l X'
\]
to be induced on the quotient by $\psi \otimes \phi: Y \otimes X \lto Y' \otimes X'$. This is clearly functorial in the sense that for morphisms $\phi': X' \lto X''$ and $\psi': Y' \lto Y''$ we have
\begin{align}
( \psi' \l \phi' ) \circ ( \psi \l \phi ) = ( \psi' \psi ) \l ( \phi' \phi )\,,
\end{align}
and $1_Y \l 1_X = 1_{Y \l X}$.
\end{definition}

The next step is to equip $Y \l X$ with a Clifford action. To do this we need to take the technical step of passing to a completion of the ring $k[y]$.

\begin{definition} For $1 \le i \le m$ set $t_i = \partial_{y_i} V$ and let $I = (t_1,\ldots,t_m) \subseteq k[y]$. Then
\[
\widehat{k[y]} = \varprojlim_r k[y]/I^r
\]
denotes the completion of $k[y]$ at the ideal $I$, that is, along the critical locus of $V$.

We also define the completed tensor product of $Y, X$ to be
\begin{equation}\label{eq:completed_tensor_product}
Y \,\widehat{\otimes}\, X = Y \otimes_{k[y]} \widehat{k[y]} \otimes_{k[y]} X
\end{equation}
\end{definition} 

\begin{remark}\label{remark:flatness} By hypothesis $k$ is noetherian, hence so is $k[y]$ and therefore the canonical ring morphism $k[y] \lto \widehat{k[y]}$ is flat. 
\end{remark}

With $\Omega^1_{k[t]/k} \cong \bigoplus_{i=1}^m k[t] dt_i$ denoting K\"ahler diferentials, there exists by \cite[Appendix B]{dm1102.2957} a $k$-linear flat connection (which is ``standard'' in the sense of \cite[Definition 8.6]{dm1102.2957})
\begin{equation}\label{eq:connection_nabla}
\nabla: \widehat{k[y]} \lto \widehat{k[y]} \otimes_{k[t]} \Omega^1_{k[t]/k}
\end{equation}
the components of which are $k$-linear operators $\partial_{t_i}: \widehat{k[y]} \lto \widehat{k[y]}$. A simple example of such a connection is given in Section \ref{example:computing_homs}. Using a chosen homogeneous basis $\{ e_a \otimes f_b \}_{a,b}$ for $Y \otimes X$ we may extend the operators $\partial_{t_i}$ extend to $k[x,z]$-linear operators
\begin{gather*}
\partial_{t_i}: Y \, \widehat{\otimes}\, X \lto Y \, \widehat{\otimes}\, X\\
e_a \otimes h(y) \otimes f_b \mapsto e_a \otimes \partial_{t_i}( h(y) ) \otimes f_b\,.
\end{gather*}

\begin{lemma} For $1 \le i \le m$ the operator
\be
[ d_{Y \otimes X}, \partial_{t_i} ] = d_{Y \otimes X} \partial_{t_i} - \partial_{t_i} d_{Y \otimes X}: Y \, \widehat{\otimes}\, X \lto Y \, \widehat{\otimes}\, X
\ee
is $k[t]$-linear and passes to a $k[x,z]$-linear operator on $Y \l X \cong \left( Y \, \widehat{\otimes}\, X \right) \otimes_{k[t]} k[t]/(t)$.
\end{lemma}
\begin{proof}
The fact that $\nabla$ is a connection over $k[t]$ means that the operators $\partial_{t_i}$ obey the Leibniz rule, that is
\[
\partial_{t_i}( t_j h ) = \partial_{t_i}( t_j ) h + t_j \partial_{t_i} h = \delta_{ij} h + t_j \partial_{t_i} h
\]
so $[ \partial_{t_i}, t_j ] = \delta_{ij}$. We calculate
\[
[[d_{Y \otimes X}, \partial_{t_i}], t_j] = - [[t_j, d_{Y \otimes X}], \partial_{t_i}] - [[\partial_{t_i}, t_j], d_{Y \otimes X}] = - [ \delta_{ij}, d_{Y \otimes X}] = 0
\]
which shows that $[d_{Y \otimes X}, \partial_{t_i}]$ is $k[t]$-linear.
\end{proof}

\begin{definition}\label{defn:atiyah} The $i$th Atiyah class is the odd $k[x,z]$-linear operator
\[
\At^{Y,X}_i = [d_{Y \otimes X}, \partial_{t_i}]: Y \l X \lto Y \l X \,.
\]
Usually we just write $\At_i$ where it will not cause confusion.
\end{definition}

This operator is the $i$th Atiyah class of $Y \,\widehat{\otimes}\, X$ relative to the ring morphism $k[x,y] \lto k[x,y,t]$ as defined in \cite[Section 9]{dm1102.2957}. A reference for Atiyah classes is \cite{buchweitz_flenner}, but we will prove the necessary properties here for the reader's convenience.

\begin{lemma}\label{lemma:atiyahclosed} $\At_i$ is $d_{Y \otimes X}$-closed, and if $\phi: X \lto X'$ and $\psi: Y \lto Y'$ are morphisms of matrix factorisations then there are $k[x,z]$-linear homotopies
\be\label{eq:atiyahclosednatural}
\At^{Y',X'}_i \circ \,( \psi \l \phi ) \simeq ( \psi \l \phi ) \circ \At^{Y,X}_i\,.
\ee
That is, the Atiyah classes are \emph{natural}.
\end{lemma}
\begin{proof}
The Atiyah class is closed since (with $d = d_{Y \otimes X}$)
\[
[d, \At_i] = [d, [d, \partial_{t_i}]] = 0\,.
\]
If $\kappa: Y \otimes X \lto Y' \otimes X'$ is any $k[x,y,z]$-linear even closed operator, 
\[
[ \kappa, [ d, \partial_{t_i} ] ] + [ \partial_{t_i}, [\kappa, d]] + [ d, [ \partial_{t_i}, \kappa] ] = 0
\]
from which we deduce that $[\partial_{t_i}, \kappa]$ is a homotopy establishing \eqref{eq:atiyahclosednatural}.
\end{proof}

The other ingredient in defining the Clifford action are the partial derivatives $\partial_{y_i}(d_X)$, as discussed in Section \ref{section:partial}. On $Y \otimes X$ we have the odd operator $\partial_{y_i}(d_X) = 1 \otimes \partial_{y_i}(d_X)$ which is a homotopy for the action of $\partial_{y_i}(V - W) = \partial_{y_i}(V)$. Hence when we pass to the quotient, $\partial_{y_i}(d_X)$ is a closed odd operator on $Y \l X$.

\begin{definition}\label{defn:cliffordaction_cut} On $Y \l X$ we define odd $k[x,z]$-linear closed operators
\begin{equation}\label{eq:intro_clifford_act1}
\ferm_i = \At_i\,, \qquad \fermc_i = - \partial_{y_i}(d_X) - \frac{1}{2} \sum_q \partial_{y_q} \partial_{y_i}(V) \At_{q}\,.
\end{equation}
\end{definition}

For a concrete example, see Section \ref{example:computing_homs}.

\begin{theorem}\label{theorem:cut_is_rep} Given matrix factorisations $Y,X$ as above, the data
\be
Y \l X = \left( Y \otimes_{k[y]} J_V \otimes_{k[y]} X, \{ \ferm_i, \fermc_i \}_{i=1}^m \right)
\ee
defines a representation of $C_m$ in the homotopy category of matrix factorisations, that is, the $\ferm_i, \fermc_i$ satisfy the Clifford relations \eqref{eq:clifford_relations}
up to homotopy.
\end{theorem}
\begin{proof}
This will follow from the considerations of Section \ref{section:theequivalence} but we feel it is worthwhile to give a direct proof, both to demonstrate how straightforward it is and in order to have written down explicit homotopies for the various relations.

With $d = d_{Y \otimes X}$ the differential on $Y \l X$, we have by the graded Jacobi identity
\begin{align*}
[\At_i, \At_j] &= [\At_i, [d, \partial_{t_j}]]\\
&= - [\partial_{t_j}, [\At_i, d]] - [d, [\partial_{t_j}, \At_i]]\\
&= [d, -[\partial_{t_j}, \At_i]]\,.
\end{align*}
Hence $[\ferm_i, \ferm_j] = [d, h_{ij}]$ where $h_{ij} = - [ \partial_{t_j}, \At_i ]$. Next we compute
\begin{align*}
[\partial_{y_i}(d_X), \At_j] &= [\partial_{y_i}(d_X), [d, \partial_{t_j}]]\\
&= -[\partial_{t_j}, [\partial_{y_i}(d_X), d]] + [d, [\partial_{t_j}, \partial_{y_i}(d_X)]]\\
&= -[\partial_{t_j}, t_i] + [d, [\partial_{t_j}, \partial_{y_i}(d_X)]]\\
&= -\delta_{ij} + [d, [\partial_{t_j}, \partial_{y_i}(d_X)]]\,.
\end{align*}
Again we set $g_{ij} = [\partial_{t_j}, \partial_{y_i}(d_X)]$ for later use. We compute that
\begin{align*}
[ \fermc_i, \ferm_j] &= \big[- \partial_{y_i}(d_X) - \frac{1}{2} \sum_q \partial_{y_q} \partial_{y_i}(V) \At_{q}, \At_j \big]\\
&= \delta_{ij} - [d, g_{ij}] - \frac{1}{2} \sum_q \partial_{y_q} \partial_{y_i}(V) [ \At_q, \At_j]\\
&= \delta_{ij} - [d, g_{ij}] - \frac{1}{2} \sum_q \partial_{y_q} \partial_{y_i}(V) [d, h_{qj}]\\
&= \delta_{ij} + \big[d, -g_{ij} - \frac{1}{2} \sum_q \partial_{y_q} \partial_{y_i}(V) h_{qj}\big]\,.
\end{align*}
Finally, using Lemma \ref{lemma:commutator_of_lambdas}
\begin{align*}
[ \fermc_i, \fermc_j ] &= \Big[ \partial_{y_i}(d_X) + \frac{1}{2} \sum_q \partial_{y_q} \partial_{y_i}(V) \At_{q}\; , \; \partial_{y_j}(d_X) + \frac{1}{2} \sum_q \partial_{y_q}\partial_{y_j}(V) \At_{q} \Big]\\
&= [ \partial_{y_i}(d_X), \partial_{y_j}(d_X) ] + \frac{1}{2}\sum_q \partial_{y_q}\partial_{y_j}(V) \big[ \partial_{y_i}(d_X), \At_q \big]\\
&\quad + \frac{1}{2}\sum_q \partial_{y_q} \partial_{y_i}(V) \big[ \At_q, \partial_{y_j}(d_X) \big]\\
&\quad + \frac{1}{4}\sum_{p,q} \partial_{y_p} \partial_{y_i}(V) \partial_{y_q} \partial_{y_j}(V) \big[ \At_p, \At_q \big]\\
&= \partial_{y_iy_j}(V) + [d, - \partial_{y_j} \partial_{y_i}(d_X)]\\
&\qquad + \frac{1}{2}\sum_q \partial_{y_q} \partial_{y_j}(V) \left( -\delta_{iq} + [d, [ \partial_{t_q}, \partial_{y_i}(d_X)]] \right)\\
&\qquad + \frac{1}{2} \sum_q \partial_{y_q} \partial_{y_i}(V) \left( -\delta_{jq} + [d, [\partial_{t_q}, \partial_{y_j}(d_X)]] \right)\\
&\qquad + \frac{1}{4}\sum_{p,q} \partial_{y_p} \partial_{y_i}(V) \partial_{y_q} \partial_{y_j}(V) [d, -[\partial_{t_q}, \At_p]]\\
&= [d, c_{ij}]
\end{align*}
where
\begin{align*}
c_{ij} &= - \partial_{y_j} \partial_{y_i}(d_X) + \frac{1}{2} \sum_q \partial_{y_q} \partial_{y_j}(V)  [\partial_{t_q}, \partial_{y_i}(d_X)]\\
&\qquad + \frac{1}{2} \sum_q \partial_{y_q} \partial_{y_i}(V) [\partial_{t_q}, \partial_{y_j}(d_X)]\\
&\qquad - \frac{1}{4} \sum_{p,q}  \partial_{y_p} \partial_{y_i}(V) \partial_{y_q} \partial_{y_j}(V)  [\partial_{t_q}, \At_p]\,.
\end{align*}
\end{proof}

\begin{lemma}\label{prop:morphismsarecliffordreps} Given morphisms $\phi: X \lto X', \psi: Y \lto Y'$, the morphism of matrix factorisations $\psi \l \phi$ is a morphism of $C_m$-representations in the homotopy category.
\end{lemma}
\begin{proof}
This follows easily from the fact that Atiyah classes are natural (Lemma \ref{lemma:atiyahclosed}) and the homotopies $\partial_{y_i}(d_X)$ are central (Lemma \ref{lemma:naturalityoflambda}).
\end{proof}

\begin{proposition}\label{prop:cutisafunctor} The cut operation defines a functor
\be\label{eq:cutisafunctor}
\L(V,U) \otimes_k \L(W,V) \lto \L(W, U)
\ee
\end{proposition}
\begin{proof}
What we have done so far in this section defines a functor
\[
\hmf( k[y,z], U - V ) \otimes \hmf( k[x,y], V - W ) \lto \hmf( k[x,z], U - W )^{\omega \bullet}
\]
and since the latter category is idempotent complete, this lifts to a functor
\[
\hmf( k[y,z], U - V )^{\omega} \otimes \hmf( k[x,y], V - W )^{\omega} \lto \hmf( k[x,z], U - W )^{\omega \bullet}\,.
\]
Explicitly, this functor is given for $(Y,e_Y) \in \L(V,U)$ and $(X,e_X) \in \L(W,V)$ by
\[
(Y, e_Y) \l (X, e_X) = (Y \l X, e_Y \l e_X)\,.
\]
Now suppose we have representations $( Y, \{ \rho_j, \rho_j^\dagger \}_{j=1}^r )$ of $C_r$ in $\L(V,U)$ and $( X, \{ \tau_l, \tau^\dagger_l \}_{l=1}^s )$ of $C_s$ in $\L(W,V)$. Here $Y,X$ may come equipped with idempotents but these are easily handled so we ignore them in what follows. Since the cut operation is functorial, we have representations of $C_r$ and $C_s$ on $Y \l X$ given by
\[
\big\{ \rho_j \l 1_X, \rho_j^\dagger \l 1_X \big\}_{j=1}^r\,, \qquad \big\{ 1_Y \l \tau_l, 1_Y \l \tau_l^\dagger \big\}_{l=1}^s\,.
\]
The operators from the action of $C_r$ anticommute with the operators from the action of $C_s$ because they act on different components of a tensor product. Moreover, both actions anticommute with the new $C_m$ action $\{ \ferm_i, \fermc_i \}_{i=1}^m$ coming from the cut, by Lemma \ref{prop:morphismsarecliffordreps}. Thus $Y \l X$ is equipped with the structure of a $C_r \otimes C_m \otimes C_s \cong C_{r+m+s}$-representation.
\end{proof}


\begin{definition} A \textsl{(super)bicategory without units} has all the data and satisfies all the conditions of a (super)bicategory, with the exception of the existence of unit $1$-morphisms, unitors, and their coherence diagram.
\end{definition}

\begin{definition} The superbicategory without units $\L$ has
\begin{itemize}
\item \textbf{Objects:} same as $\LG$, that is, potentials $(x, W)$.
\item \textbf{$1$- and $2$-morphisms:} defined by \eqref{eq:defnluv}.
\item \textbf{Composition:} defined by the cut operation \eqref{eq:cutisafunctor}.
\item \textbf{Associators:} given a triple of composable $1$-morphisms
\begin{equation}\label{eq:composable_triple}
\xymatrix@C+2pc{
(q, Q) & (y, U) \ar[l]_-{Z} & (z,V) \ar[l]_-{Y} & (x,W) \ar[l]_-{X}
}\,.
\end{equation}
with $V \in k[z_1,\ldots,z_m]$ and $U \in k[y_1,\ldots,y_p]$ the isomorphism
\begin{gather*}
\alpha: Z \l (Y \l X) \lto (Z \l Y) \l X,\\
z \otimes (y \otimes x) \mapsto (z \otimes y) \otimes x
\end{gather*}
is $C_m \otimes C_p$-linear (see Lemma \ref{lemma:associator} below) and is the associator for $\L$.
\end{itemize}
\end{definition}

\begin{lemma}\label{lemma:associator} The morphism $\alpha: (Z \l Y) \l X \lto Z \l (Y \l X)$ is $C_m \otimes C_p$-linear.
\end{lemma}
\begin{proof}
We prove the $C_p$-linearity of $\alpha$, the proof for $C_m$ is similar. The algebra $C_p$ acts on the cut between $Z$ and $Y$ involving the variables $y_1,\ldots,y_p$. The action on $Z \l ( Y \l X ) = Z \otimes J_U \otimes (Y \otimes J_V \otimes X)$ is via $1 \otimes (\partial_{y_i}(d_Y) \otimes 1)$ and the Atiyah classes of 
\[
M = Z \otimes_{k[y]} \widehat{k[y]} \otimes_{k[y]} \big( Y \otimes_{k[z]} J_V \otimes_{k[z]} X \big)
\]
relative to the ring morphism $k[q,x] \lto k[q,x] [ \partial_{y_1} U, \ldots, \partial_{y_p} U ]$. Specifically, the generator $\ferm_i$ of $C_p$ is the commutator $[d_{M}, \partial_s]$ where $s = \partial_{y_i} U$. But this operator can be identified with an Atiyah class of
\begin{equation}\label{eq:assoc_proof_M}
Z \otimes_{k[y]} \widehat{k[y]} \otimes_{k[y]} \big( Y \otimes_{k[z]} \widehat{k[z]} \otimes_{k[z]} X \big)
\end{equation}
relative to the ring morphism
\begin{equation}\label{eq:assoc_proof_M2}
k[q,x] \lto k[q,x] [ \partial_{y_1} U, \ldots, \partial_{y_p} U, \partial_{z_1} V, \ldots, \partial_{z_m} V ]\,.
\end{equation}
By the same argument the action of $C_p$ on $(Z \l Y ) \l X$ is via $(1 \otimes \partial_{y_i}(d_Y)) \otimes 1$ and the Atiyah classes of $\big( Z \otimes_{k[y]} \widehat{k[y]} \otimes_{k[y]} Y \big) \otimes_{k[z]} \widehat{k[z]} \otimes_{k[z]} X$ relative to the same ring morphism \eqref{eq:assoc_proof_M2}. It is therefore clear that the associator $\alpha$ identifies these two actions.
\end{proof}

\begin{proposition} Thus defined $\L$ is a superbicategory without units.
\end{proposition}
\begin{proof}
The only thing to check is the coherence of the associator, which is trivial.
\end{proof}

\begin{remark}\label{remark:relation_to_toby_paper} The cut operation presented here refines earlier work in \cite{dm1102.2957}. To see the connection, recall the idempotent $e_m \in C_m$ of Definition \ref{defn:idempotent_e} which corresponds to projection onto $k \cdot 1 \subset S_m$. The projector onto top degree $k \cdot \theta_1 \cdots \theta_m \subset S_m$ is 
\[
e_m' = \fermc_1 \cdots \fermc_m \ferm_m \cdots \ferm_1\,.
\]
Observe that each of the $\fermc_i$'s acts as $- \partial_{y_i}(d_X)$ because all of the $\ferm_j$'s to the right. Hence
\begin{align*}
e_m' &= (-1)^{\binom{m}{2}} \fermc_1 \cdots \fermc_m \ferm_1 \cdots \ferm_m \\
&= (-1)^{\binom{m+1}{2}} \lambda_1 \cdots \lambda_m \At_1 \cdots \At_m\,.
\end{align*}
which is precisely the idempotent which appears in \cite[Corollary 10.4]{dm1102.2957}.

In \emph{ibid.} we constructed finite models of convolutions $Y \circ X$ one pair $(Y,X)$ at time. To make this assignment of finite models coherently for all $1$-morphisms in $\LG$ simultaneously it is natural to identify the Clifford actions which underlie the idempotents, as we do here.
\end{remark}

\begin{remark} There is a question of whether or not the $C_m$-action on $Y \l X$ is canonical, since it depends on the choice of connection $\nabla$ and on a choice of homogeneous basis for $X$ and $Y$. But Atiyah classes are independent, up to homotopy, of the choice of connection.
\end{remark}

\section{The equivalence of $\L$ and $\LG$}\label{section:equivllg}

We prove that the cut functor defined in Section \ref{section:lg_cut_system} models composition in $\LG$ by providing an equivalence of this superbicategory with $\L$. This is done by presenting the cut $Y \l X$ with its Clifford action as the solution of the problem of finding a finite model for the matrix factorisation $Y \otimes X$ over $k[x,z]$.

As was mentioned in Section \ref{section:superbicatLG}, $Y \otimes X$ is a matrix factorisation of infinite rank over $k[x,z]$. The obvious notion of a ``finite model'' would be a finite rank matrix factorisation to which $Y \otimes X$ is homotopy equivalent, but providing such a finite model generically seems impossible. We propose a different notion of finite model based on Clifford representations, and solve the problem in this setting. The solution we give here builds on earlier progress on this problem with Dyckerhoff \cite{dm1102.2957}. 

\begin{setup}\label{setupforfusion} The situation is as follows:
\begin{itemize}
\item $(x,W), (y, V), (z,U)$ are objects of $\LG$ with $k[y] = k[y_1,\ldots,y_m]$.
\item $Y$ is a finite rank matrix factorisation of $U - V$ over $k[y,z]$.
\item $X$ is a finite rank matrix factorisation of $V - W$ over $k[x,y]$.
\end{itemize}
\end{setup}

The strategy is to consider the direct sum of (shifted) copies
\begin{equation}\label{eq:larger_object}
S_m \otimes_k ( Y \otimes X ) = (Y \otimes X) \oplus (Y \otimes X)[1] \oplus (Y \otimes X) \oplus \cdots
\end{equation}
where $S_m$ is the exterior algebra viewed as a $\nZ_2$-graded $k$-module
\[
S_m = \bigwedge( k[1]^{\oplus m} )\,.
\]
We find a finite model of this larger object, and in addition record a method of extracting the subobject $Y \otimes X$. Some necessary background is presented in Section \ref{section:homolog_fix} and Section \ref{section:cliffordactkos}, and finally the full construction is presented in Section \ref{section:theequivalence}. The overall strategy is:

\begin{strat}\label{strategy} We show there is a differential $d_K$ ($K$ stands for Koszul) such that
\begin{itemize}
\item[1)] The complex $( S_m \otimes_k ( Y \otimes X ), d_K )$ has a finite model: there is a homotopy equivalence
\be
(Y \l X, 0) \cong \big( S_m \otimes_k ( Y \otimes X ), d_K \big)
\ee
 over $k[x,z]$ where $Y \l X$ is the graded module from Section \ref{section:lg_cut_system}.
\item[2)] The perturbation lemma (Section \ref{section:homolog_fix}) promotes this to a homotopy equivalence
\be\label{eq:woofeif}
(Y \l X, d_{Y \otimes X}) \cong \big( S_m \otimes_k ( Y \otimes X ), d_K + d_{Y \otimes X}\big)
\ee
by mixing in the differential $d_{Y \otimes X}$ on both sides.
\item[3)] Using Section \ref{section:cliffordactkos} we show there is a homotopy equivalence
\be\label{eq:woofeif2}
\big( S_m \otimes_k ( Y \otimes X ), d_K + d_{Y \otimes X}\big) \cong \big( S_m \otimes_k ( Y \otimes X ), d_{Y \otimes X}\big)\,.
\ee
\item[4)] Combining \eqref{eq:woofeif} and \eqref{eq:woofeif2} we have the desired finite model
\be\label{eq:iso_of_clifford_modelfinite}
(Y \l X, d_{Y \otimes X}) \cong ( S_m \otimes_k ( Y \otimes X ), d_{Y \otimes X})\,.
\ee
There is a canonical action of the Clifford algebra $C_m$ on $S_m$ and thus on the right hand side, which ``remembers'' how to extract $Y \otimes X$ from this larger object. Transferring this action yields the operators $\{ \ferm_i, \fermc_i \}_{i=1}^m$ on $Y \l X$ from Section \ref{section:lg_cut_system}.
\end{itemize}
\end{strat}

In conclusion, there is an isomorphism of Clifford representations \eqref{eq:iso_of_clifford_modelfinite} in the homotopy category of finite rank matrix factorisations of $U - W$ over $k[x,z]$, and this is our solution to the problem of finding a finite model of $Y \otimes X$. Once we have established all of this, proving that $\L$ and $\LG$ are equivalent bicategories will be straightforward.

\subsection{Homological perturbation and fixed points}\label{section:homolog_fix}

A fundamental role in this paper is played by the homological perturbation lemma. While this result is usually stated for complexes the standard results generalise immediately to linear factorisations, and in this section we collect these standard results. But we begin by recalling the motivating problem for which the perturbation lemma is the solution.

Let $Q$ be a complex of vector spaces (possibly infinite dimensional) with finite dimensional cohomology, so that $Q$ contains in some sense only a ``finite'' amount of information. One way of making this finiteness manifest is to simply list the cohomology groups of $Q$, but a more flexible and categorical solution is to find a finite representative in the homotopy equivalence class of $Q$.

A \emph{deformation retract} of complexes is a pair of morphisms
\begin{equation}\label{eq:example_defo_retract}
\xymatrix@C+2pc{
C \ar@<-1ex>[r]_\sigma & Q, \ar@<-1ex>[l]_\pi
} \quad \phi
\end{equation}
and a degree $-1$ operator $\phi: Q \lto Q$ satisfying $\pi \sigma = 1_C, \sigma \pi = 1_Q - [d_Q, \phi]$ so that $\pi, \sigma$ are mutually inverse homotopy equivalences. A deformation retract is \emph{strong} if it satisfies some additional conditions, listed below. The point is that if $C$ is finitely generated, then \eqref{eq:example_defo_retract} gives the desired finite representative in the homotopy equivalence class of $Q$. 

\begin{example}\label{example:koszulsplit} Consider the following diagram of complexes of $k$-vector spaces
\[
\xymatrix@C+2pc@R+1pc{
0 \ar[d] \ar@<-1ex>[r] & k[x] \ar@<-1ex>[l] \ar[d]^-{x}\\
k \ar@<-1ex>[r]_-{\sigma} & k[x] \ar@<-1ex>[l]_-{\pi} \ar@/_3pc/@{.>}[u]_-{\phi}
}
\]
in which the first column is the complex $C$ which has $k$ concentrated in degree zero, and the second column $Q$ is the Koszul complex on $x$. The map $\pi: k[x] \lto k$ sends every polynomial to its constant term, while $\sigma: k \lto k[x]$ is the inclusion of the constants. The operator $\phi$ is defined by $\phi( x^n ) = x^{n-1}$ and $\phi(1) = 0$. Together these maps are a deformation retract giving $k$ as the ``finite model'' of $Q$.
\end{example}

Suppose now that a deformation retract computing a finite representative for $Q$ cannot be found immediately, but that we know how to decompose the differential as $d_Q = d + \tau$ in such a way that there is a finite representative for the complex $(Q,d)$. The perturbation lemma allows us, with some conditions, to ``mix in'' the operator $\tau$ to obtain a finite model for the original complex $Q$.
\\

Let $R$ be a commutative ring and $W \in R$. Everything we say applies, in the case where $W = 0$, to both $\mathbb{Z}_2$-graded and $\mathbb{Z}$-graded complexes.  For our purposes the reformulation of the perturbation lemma in terms of fixed points by Barnes and Lambe is more natural, so we emphasise the ``splitting homotopies'' of \cite{barneslambe}.

\begin{definition} A \emph{splitting homotopy} on a linear factorisation $(A,d)$ of $W$ is a degree $-1$ operator $\phi$ on $A$ satisfying
\begin{itemize}
\item[(i)] $\phi^2 = 0$,
\item[(ii)] $\phi d \phi = \phi$.
\end{itemize}
A \emph{morphism} $(A,d, \phi) \lto (A,d',\phi')$ of splitting homotopies is a morphism $\alpha: A \lto A'$ of linear factorisations satisfying $\phi' \alpha = \alpha \phi$.
\end{definition}

The condition (ii) says that $\phi$ is a \emph{fixed point} of the operator $F(x) = x d x$.

\begin{definition} A \emph{strong deformation retract} of linear factorisations of $W$ is a pair of morphisms of linear factorisations
\begin{equation}\label{eq:defn_sdr}
\xymatrix@C+2pc{
(M,d) \ar@<-1ex>[r]_\sigma & (A,d), \ar@<-1ex>[l]_\pi
} \quad \phi
\end{equation}
together with a degree $-1$ operator $\phi: A \lto A$ satisfying
\begin{itemize}
\item[(i)] $\pi \sigma = 1_M$,
\item[(ii)] $\sigma \pi = 1_A - [d, \phi]$,
\item[(iii)] $\phi^2 = 0$,
\item[(iv)] $\phi \sigma = 0$,
\item[(v)] $\pi \phi = 0$.
\end{itemize}
A \emph{morphism} of strong deformation retracts is a a pair of morphisms of linear factorisations $\alpha: (M,d) \lto (M',d')$ and $\beta: (A,d) \lto (A',d')$ satisfying $\beta \sigma = \sigma' \alpha, \alpha \pi = \pi' \beta$, that is, making the two squares implicit in the diagram
\[
\xymatrix@C+2pc{
(M,d) \ar[d]_{\alpha} \ar@<-1ex>[r]_\sigma & (A,d)\ar[d]^\beta \ar@<-1ex>[l]_\pi\\
(M',d') \ar@<-1ex>[r]_{\sigma'} & (A',d') \ar@<-1ex>[l]_{\pi'}
}
\]
commute, and in addition satisfying $\phi' \beta = \beta \phi$.
\end{definition}

\begin{remark}\label{remark:strongdefosigns} We follow the conventions of \cite{barneslambe} in defining strong deformation retracts; up to a sign this the same as the special deformation retracts of \cite{crainic} and \cite{dm1102.2957}. Namely, a strong deformation retract in the above sense is the same as a special deformation retract in the sense of \emph{ibid.} with $h = - \phi$.
\end{remark}

\begin{lemma}\label{lemma:equivocate} There is an equivalence between the category of splitting homotopies $\cat{C}_1$ and the category of  strong deformation retracts $\cat{C}_2$.
\end{lemma}
\begin{proof}
We briefly recall the construction from \cite[p.883]{barneslambe}, which is not stated in terms of categories but is obviously functorial. Given a strong deformation retract \eqref{eq:defn_sdr} the data $(A, \phi)$ is a splitting homotopy. In the reverse direction, if $(A, \phi)$ is a splitting homotopy then $e = 1_A - [d, \phi]$ is idempotent and we define the linear factorisation $M = \operatorname{im}(e)$ with the associated projection $\pi$ and inclusion $\sigma$. Then this is a strong deformation retract, together with the original $\phi$.
\end{proof}

Here is the problem that the formalism is designed to solve: suppose $(A,d)$ is a linear factorisation of $W$ and that $\tau$ is a degree $+1$ operator on $A$ such that $(A, d + \tau)$ is a linear factorisation of $V$ (possibly different to $W$). This $\tau$ is called the \emph{perturbation}.

\begin{definition} Given a splitting homotopy $\phi$ on $(A,d)$ and a perturbation $\tau$ in the above sense, the \emph{transference problem} is to find a splitting homotopy $\phi'$ on $(A, d + \tau)$ such that $\im(e) \cong \im(e')$ as graded $R$-modules, where $e = 1_A - [d, \phi]$ and $e' = 1_A - [d + \tau, \phi']$. 
\end{definition}

That is, with $\xi = d + \tau$ the problem is to find a fixed point of the operator
\[
F(x) = x \xi x
\]
among operators with $x^2 = 0$ satisfying the boundary condition $\im(1_A - [d,x]) \cong \im(e)$. If $\phi \tau$ has finite order, i.e. $(\phi \tau)^m = 0$ for some $m$, then
\begin{equation}\label{eq:formula_phi_infty}
\phi_\infty = \sum_{m \ge 0} (-1)^m (\phi \tau)^m \phi = \phi - \phi \tau \phi + \phi \tau \phi \tau \phi - \cdots
\end{equation}
is a solution of this fixed point problem:

\begin{theorem}[Perturbation lemma]\label{theorem:pertlemma} Suppose $\phi \tau$ has finite order. Then $\phi_\infty$ is a splitting homotopy on the linear factorisation $(A, \xi = d + \tau)$ and satisfies the isomorphism condition in the transference problem.
\end{theorem}
\begin{proof}
The proof for complexes \cite[p.886]{barneslambe} goes through unchanged for linear factorisations, with some minor modifications that are elaborated in \cite[\S 2.5]{lgdual}.
\end{proof}

The more common statement of the basic perturbation lemma involves the deformation retract corresponding to the splitting homotopy $\phi_\infty$ under the equivalence of Lemma \ref{lemma:equivocate}. The statement is that given a splitting homotopy $\phi$ as above with associated deformation retract \eqref{eq:defn_sdr} there is a strong deformation retract
\begin{equation}\label{eq:perturbedsdr}
\xymatrix@C+2pc{
(M,d_\infty) \ar@<-1ex>[r]_{\sigma_\infty} & (A, \xi = d + \tau), \ar@<-1ex>[l]_{\pi_\infty}
} \quad \phi_\infty
\end{equation}
where $C = \tau( 1 + \phi \tau )^{-1} = \sum_{m \ge 0} (-1)^m \tau (\phi \tau)^m$ and
\begin{align*}
\sigma_\infty &= \sigma - \phi C \sigma,\\
\pi_\infty &= \pi - \pi C \phi,\\
d_\infty &= d + \pi C \sigma\,.
\end{align*}
As explained in \cite{barneslambe}, to derive these formulas one has only to notice that the sub-complex in the deformation retract associated to $\phi_\infty$ can, by the isomorphism
\[
\im(1 - [\xi, \phi_\infty]) \cong \im(e) = M
\]
be identified with $M$, and the induced maps are the $\sigma_\infty, \pi_\infty, d_\infty$ given above. It is also possible to give a direct proof as in \cite{crainic}. In applications the most important output of the construction are the formulas \eqref{eq:formula_phi_infty} for $\phi_\infty$ and
\be
\sigma_\infty = \sum_{m \ge 0} (-1)^m (\phi \tau)^m \sigma\,,
\ee
which follows easily from the above.

\subsection{Clifford actions and Koszul complexes}\label{section:cliffordactkos}

Let $R$ be a $k$-algebra and $(X,d)$ a complex of $R$-modules on which $t_1,\ldots,t_m \in R$ act null-homotopically. Tensoring $X$ with the Koszul complex $K(t_i)$ yields the same complex as taking the cone of $t_i \cdot 1_X \simeq 0: X \lto X$, and therefore we have a homotopy equivalence
\be
K(t_i) \otimes_R X \cong \operatorname{cone}( t_i \cdot 1_X ) \cong X \oplus X[1] \cong ( k \oplus k[1] ) \otimes_k X\,.
\ee
Iterating this operation we obtain a homotopy equivalence with $K = K(t_1,\ldots,t_m)$,
\be\label{eq:cliffordactkos_1}
K \otimes_R X \cong \bigwedge( k[1]^{\oplus m} ) \otimes_k X
\ee
where the left hand side is a complex with differential $d_K + d$ and the right hand side has only the differential $d$. On the right hand side we have an obvious action of
\be
C_m = \End_k\left( \bigwedge( k[1]^{\oplus m} ) \right)
\ee
and the purpose of this section is to compute the corresponding action of $C_m$ via closed odd endomorphisms of the complex $K \otimes_R X$. In fact we will do this more generally:

\begin{setup} Let $R$ be a $k$-algebra and write $\otimes$ for $\otimes_R$. Let $(X,d)$ be a linear factorisation of $W \in R$ and $t_1,\ldots,t_m$ a sequence of elements of $R$ acting null-homotopically on $X$. We choose odd $R$-linear operators $\lambda_i$ on $X$ with $[\lambda_i, d] = t_i \cdot 1_X$ for $1 \le i \le m$.
\end{setup}

\begin{remark} All equalities in this section mean equalities of linear maps. If we mean homotopy, we will explicitly write $\simeq$.
\end{remark}

We introduce formal variables $\theta_1,\ldots,\theta_m$ of odd degree and set
\[
S_m = \bigwedge\left( k \theta_1 \oplus \cdots \oplus k \theta_m \right)\,.
\]
On this $\nZ$-graded $k$-module there are canonical odd $k$-linear operators $\theta_i, \theta_i^*$ defined as in Definition \ref{defn:contraction}. The Koszul complex $K$ of the sequence $t_1,\ldots,t_m$ is defined by
\begin{equation}\label{defn:koszul}
K = S_m \otimes_k R, \qquad d_K = \sum_{i=1}^m t_i \theta_i^*\,.
\end{equation} 
The tensor products $K \otimes X = (K \otimes X, d + d_K)$ and $S_m \otimes_k X = ( S_m \otimes_k X, d )$ are both linear factorisation of $W$ with isomorphic underlying graded modules. On this underlying graded module $K \otimes X \cong S_m \otimes_k X$ we have the odd $R$-linear operators $\theta_k^* = \theta_k^* \otimes 1$ and $\theta_k = \theta_k \otimes 1$ and we introduce the following even $R$-linear operators
\begin{align*}
\delta &= \sum_{i=1}^m \lambda_i \theta_i^*,\\
\exp(-\delta) &= \sum_{n \ge 0} (-1)^n \frac{1}{n!} \delta^n\,.
\end{align*}
Observe that $\delta$ is nilpotent so the exponential makes sense.

\begin{proposition}\label{prop:equivalencekoszul} There are mutually inverse isomorphisms of linear factorisations
\be\label{eq:equivalencekoszul}
\xymatrix@C+5pc{ K \otimes X\; \ar@<1ex>[r]^-{ \exp(\delta) } & \;S_m \otimes_k X \ar@<1ex>[l]^-{ \exp(-\delta) } }
\ee
where the left hand side has differential $d + d_K$ and the right hand side has $d$.
\end{proposition}

We break the proof into a series of lemmas:

\begin{lemma}\label{lemma:pert1} $[d, \delta^n] = n \delta^{n-1} d_K$ for $n \ge 1$.
\end{lemma}
\begin{proof}
The case of $n = 1$ is clear
\[
[d, \delta] = \sum_i [d, \lambda_i] \theta_i^* = d_K\,.
\]
For $n > 1$ we use the case $n = 1$ to show that
\[
[ d, \delta^n ] = \sum_{i=0}^{n-1} \delta^i [d, \delta] \delta^{n-i-1} = \sum_{i=0}^{n-1} \delta^i d_K \delta^{n-i-1} = \sum_{i=0}^{n-1} \delta^i \delta^{n-i-1} d_K = n \delta^{n-1} d_K\,.
\]
\end{proof}

\begin{lemma}\label{lemma:pert2} $[ d, \exp(-\delta) ] = - \exp(-\delta) d_K$ and $[ d, \exp(\delta) ] = \exp(\delta) d_K$.
\end{lemma}
\begin{proof}
We prove the first identity using Lemma \ref{lemma:pert1}, the second is the same:
\begin{align*}
[d, \exp(-\delta)] &= \sum_{n \ge 1} \frac{(-1)^n}{n!} [ d, \delta^n ]\\
&= \sum_{n \ge 1} \frac{(-1)^n}{n!} n \delta^{n-1} d_K\\
&= - \sum_{n \ge 0} \frac{(-1)^n}{n!} \delta^n d_K\\
&= - \exp(-\delta) d_K
\end{align*}
\end{proof}

\begin{proof}[Proof of Proposition \ref{prop:equivalencekoszul}]
It suffices to show $\exp(-\delta)$ is a morphism of linear factorisations. But using Lemma \ref{lemma:pert2}
\[
(d + d_K) \exp(-\delta) - \exp(-\delta) d = [d, \exp(-\delta)] + d_K \exp(-\delta) = 0
\]
and similarly for the other equation.
\end{proof}

\begin{definition}\label{defn:transfer_contract} Let $\gamma$ be a homogeneous $R$-linear operator on $S_m \otimes_k X$. We define the \textsl{transfer} $\cat{T}(\gamma)$ to be the homogeneous $R$-linear operator on $K \otimes X$ induced by $\gamma$ using the equivalence \eqref{eq:equivalencekoszul}, that is,
\be
\cat{T}(\gamma) = \exp(-\delta) \gamma \exp(\delta)\,.
\ee
\end{definition}

\begin{remark} We do not necessarily require that $\gamma$ be closed, but since
\be
[\cat{T}(\gamma), d_K + d] = \cat{T}( [\gamma, d] )
\ee
if $\gamma$ is a closed operator then so is $\cat{T}(\gamma)$. 
\end{remark}

Of particular interest are the transfers of the generators $\theta_i, \theta_i^*$ of the Clifford action. The transfer of the contraction operator $\theta_i^*$ is easy:

\begin{lemma} $\cat{T}(\theta_i^*) = \theta_i^*$.
\end{lemma}
\begin{proof}
Since $\theta_i^*$ anticommutes with $\delta$.
\end{proof}

However the calculation of $\cat{T}(\theta_i)$ is much more involved. To compute $\cat{T}(\gamma)$ the strategy is to try and commute $\gamma$ past $\exp(-\delta)$ and see what happens. To this end we must first compute $[\gamma, \delta^m]$. We begin the calculations with some general statements about graded commutators. 

To this end let $A$ be a $\nZ$-graded $\nQ$-algebra and $a,b,c$ homogeneous elements.

\begin{lemma} We have
\begin{align}
[ab,c] &= a[b,c] + (-1)^{|b||c|} [a,c]b\,,\\
[a,bc] &= [a,b]c + (-1)^{|a||b|} b[a,c]\,. \label{eq:commutator_with_product}
\end{align}
\end{lemma}

We adopt the following notation for iterated commutators

\begin{definition} The iterated commutator $[a,b]^{(n)}$ is defined by $[a,b]^{(0)} = b$ and
\[
[a,b]^{(n+1)} = [a, [a,b]^{(n)}]\,.
\]
Observe that $[a,b]^{(n)}$ is a linear function of $b$, but not $a$.
\end{definition}

\begin{lemma}\label{lemma:woop3} For $n \ge 1$ we have
\[
[a,b^n] = \sum_{i=0}^{n-1} (-1)^{i|a||b|} b^i [a,b] b^{n-i-1}\,.
\]
\end{lemma}
\begin{proof}
Easy proof by induction using \eqref{eq:commutator_with_product}.
\end{proof}

\begin{lemma}\label{lemma:woop4} For $n \ge 1$ we have
\be
a^n [a, b] = \sum_{\substack{0 \le i,j \le n \\ i+j=n}} \binom{n}{i} [a,b]^{(i+1)} a^j\,.
\ee
\end{lemma}
\begin{proof}
By induction, with the case $n = 1$ being clear. Using the inductive hypothesis
\[
a^{n+1} [a, b] = \sum_{i+j=n} \binom{n}{i} a [a,b]^{(i+1)} a^j = \sum_{i+j=n} \binom{n}{i} \Big[ [a,b]^{(i+2)} a^j + [a,b]^{(i+1)} a^{j+1} \Big]\,.
\]
The powers of $a$ range from $0$ to $n+1$ and the coefficient of $a^j$ for $1 \le j \le n$ is
\[
\binom{n}{n-j} [a,b]^{(n-j+2)} + \binom{n}{n-j+1} [a,b]^{(n-j+2)} = \binom{n+1}{n+1-j} [a,b]^{(n-j+2)}\,.
\]
The coefficient of $a^0$ is $[a,b]^{(n+2)}$ and the coefficient of $a^{n+1}$ is $[a,b]^{(1)}$, so we are done.
\end{proof}

\begin{lemma}\label{lemma_iteratedcomm} Let odd elements $a_i, b_i \in A$ be given for $1 \le i \le m$ such that $[a_i, b_j] = 0$ for all $1 \le i,j \le m$ and define
\[
D = \sum_{i=1}^m a_i b_i\,.
\]
Suppose $[b_i,c] = 0$ for $1 \le i \le m$. Then for $n \ge 0$
\be\label{eq:lemma_iteratedcomm}
[D, c]^{(n)} = (-1)^{n|c|} \sum_{1 \le q_1,\ldots,q_n \le m} [a_{q_n}, [a_{q_{n-1}}, \ldots, [ a_{q_1}, c] \cdots ]] b_{q_1} \cdots b_{q_n}\,.
\ee
\end{lemma}
\begin{proof}
For convenience we introduce, for any tuple of indices $\bold{q} = (q_1,\ldots,q_n)$ of integers $1 \le q_i \le m$, the notation
\begin{gather}
\tau_{\bold{q}} = [ a_{q_n}, [ a_{q_{n-1}}, \ldots [ a_{q_1}, c ] \cdots ] \label{defn:tau}\\
b_{\bold{q}} = b_{q_1} \cdots b_{q_n} \label{defn:bq}\,.
\end{gather}
For $n = 0$ the right hand side of \eqref{eq:lemma_iteratedcomm} is $c$ by convention, so the statement is trivial. We compute using the inductive hypothesis that
\begin{align*}
[D, c]^{(n+1)} &= [ D, [D, c]^{(n)} ]\\
&= \big[D, (-1)^{n|c|} \sum_{\bold{q}} \tau_{\bold{q}} b_{\bold{q}}\big]\\
&= (-1)^{n|c|} \sum_{\bold{q}} \sum_i [ a_i b_i , \tau_{\bold{q}} b_{\bold{q}}]\\
&= (-1)^{n|c|} \sum_{\bold{q}} \sum_i \left( a_i b_i \tau_{\bold{q}} b_{\bold{q}} - \tau_{\bold{q}} b_{\bold{q}} a_i b_i \right)\\
&= (-1)^{n|c|} \sum_{\bold{q}} \sum_i \left( (-1)^{|c|} a_i \tau_{\bold{q}} b_{\bold{q}} b_i - (-1)^n \tau_{\bold{q}} a_i b_{\bold{q}} b_i \right)\\
&= (-1)^{(n+1)|c|} \sum_{\bold{q}, i} [ a_i, \tau_{\bold{q}} ] b_{\bold{q}} b_i
\end{align*}
which completes the proof.
\end{proof}

Eventually we are going to substitute $a_i = \lambda_i, b_i = \theta_i^*$ so that $D = \delta$, and we will take further $c = \gamma$. In this notation, the aim of all of this is to compute the commutators $[c, D^n]$ and then $[c, \exp(-D)]$. Since most of the calculation can be done in $A$, we stick to the general setting as long as possible, and maintain the notation $a_i, b_i, c, D$ and \eqref{defn:tau}, \eqref{defn:bq} of the previous lemma.

\begin{lemma}\label{lemma:commutator_wha} For $n \ge 1$
\[
[ c, D^n ] = \sum_{i=0}^{n-1} (-1)^{(i+1)|c| + 1} \sum_{i \le p \le n-1} \binom{p}{i} \sum_{\bold{q} = (q_1,\ldots,q_{i+1})} \tau_{\bold{q}} b_{\bold{q}} D^{n-i-1}\,.
\]
\end{lemma}
\begin{proof} 
By Lemma \ref{lemma:woop3} we have
\begin{equation}\label{eq:catfish}
[ c, D^n ] = \sum_{p=0}^{n-1} D^p [ c, D ] D^{n-p-1} = -\sum_{p=0}^{n-1} D^p [ c, D ] D^{n-p-1}\,.
\end{equation}
Then using Lemma \ref{lemma:woop4} this becomes
\be
-\sum_{p=0}^{n-1} \sum_{i+j=p} \binom{p}{i} [D, c]^{(i+1)} D^{n-p-1+j}
\ee
which by Lemma \ref{lemma_iteratedcomm} equals
\be
\sum_{p=0}^{n-1} \sum_{i+j=p} \binom{p}{i} (-1)^{(i+1)|c| + 1} \sum_{\bold{q} = (q_1,\ldots,q_{i+1})} \tau_{\bold{q}} b_{\bold{q}} D^{n-i-1}
\ee
which is what we needed to show.
\end{proof}

Putting this all together:

\begin{proposition}\label{prop:ccommutesexpD} Let $A$ be a $\nZ_2$-graded $\nQ$-algebra and
\begin{itemize}
\item Let $a_i, b_i \in A$ be odd elements with $[a_i,b_j] = 0$ for $1 \le i,j \le m$.
\item Assume that $D = \sum_{i=1}^m a_i b_i$ has finite order, that is, $D^r = 0$ for some $r > 0$.
\item Let $c \in A$ be an odd element, with $[b_i, c] = 0$ for $1 \le i \le m$.
\end{itemize}
Then
\be\label{eq:ccommutesexpD}
[c, \exp(-D)] = - \sum_{n \ge 1} \frac{1}{n!} \sum_{q_1,\ldots,q_n} [ a_{q_n}, [ a_{q_{n-1}}, \ldots [ a_{q_1}, c ] \cdots ] b_{q_1} \cdots b_{q_n} \exp(-D)\,.
\ee
\end{proposition}
\begin{proof}
Using Lemma \ref{lemma:commutator_wha}
\begin{align*}
[c, \exp(-D)] &= \sum_{n \ge 0} \frac{(-1)^n}{n!} [c, D^n]\\
&= \sum_{n \ge 0} \frac{(-1)^{n+i}}{n!} \sum_{i=0}^{n-1} \sum_{i \le p \le n -1} \binom{p}{i} \sum_{\bold{q} = (q_1,\ldots,q_{i+1})} \tau_{\bold{q}} b_{\bold{q}} D^{n-i-1}\,.
\end{align*}
Rewriting this sum in terms of the variable $t = n - i - 1$ yields
\[
- \sum_{t \ge 0} \sum_{i \ge 0} \frac{(-1)^{t}}{(t+i+1)!} \sum_{i \le p \le t+i}  \binom{p}{i} \sum_{\bold{q} = (q_1,\ldots,q_{i+1})} \tau_{\bold{q}} b_{\bold{q}} D^{t}\,.
\]
Using Lemma \ref{lemma:combin} below this may be rewritten as
\[
- \sum_{t \ge 0} \sum_{i \ge 0} \frac{(-1)^t}{t!} \frac{1}{(i+1)!} \sum_{\bold{q} = (q_1,\ldots,q_{i+1})} \tau_{\bold{q}} b_{\bold{q}} D^{t}
\]
which clearly agrees with \eqref{eq:ccommutesexpD}.
\end{proof}

\begin{lemma}\label{lemma:combin} For integers $a, b \ge 0$
\begin{equation}\label{eq:combin}
\frac{b!}{(a+b+1)!} \sum_{a \le n \le a + b} \binom{n}{a} = \frac{1}{(a+1)!}
\end{equation}
\end{lemma}
\begin{proof}
Denoting the left hand side of \eqref{eq:combin} by $\Omega_{a,b}$ the proof is an induction on $b$, i.e. by computing $\Omega_{a,b+1}$ in terms of $\Omega_{a,b}$.
\end{proof}

Now let us return to the original problem, of calculating the transfer $\cat{T}(\gamma)$ from Definition \ref{defn:transfer_contract}. We will only need this in the following special case:

\begin{corollary}\label{corollary_transfergamma} Let $\gamma$ be an odd $R$-linear operator on $X$. Then
\be\label{eq:transfer_gamma}
\cat{T}(\gamma) = \gamma + \sum_{n \ge 1} \sum_{q_1,\ldots,q_n}  \frac{1}{n!} [\lambda_{q_n}, [\lambda_{q_{n-1}}, \ldots [ \lambda_{q_1}, \gamma] \cdots ]] \theta^*_{q_1} \ldots \theta^*_{q_n}
\ee
is the tranferred operator on $K \otimes X$, where the $q_j$'s range over $1, \ldots, m$.
\end{corollary}
\begin{proof}
This is immediate from Proposition \ref{prop:ccommutesexpD}, substituting $a_i = \lambda_i, b_i = \theta_i^*,c = \gamma$, since
\[
\exp(-D) c \exp(D) = c - [c, \exp(-D)] \exp(D)\,.
\]
\end{proof}


\begin{remark} Observe that since the $\theta_i^*$ square to zero, only sequences of distinct indices $q_1,\ldots,q_n$ contribute in the sum \eqref{eq:transfer_gamma}. The low order terms are
\be
\cat{T}(\gamma) = \gamma + \sum_q [\lambda_q, \gamma] \theta^*_q + \frac{1}{2} \sum_{p,q} [\lambda_p, [\lambda_q, \gamma]] \theta_q^*\theta_p^* + \cdots
\ee
\end{remark}

Finally, we compute the transfer of the operator $\theta_i$ on $S_m \otimes_k X$:

\begin{theorem}\label{theorem:psik} We have
\begin{equation}\label{eq:psik}
\cat{T}(\theta_i) = \theta_i - \sum_{n \ge 0} \sum_{q_1,\ldots,q_n} \frac{1}{(n+1)!} [ \lambda_{q_n}, [ \lambda_{q_{n-1}}, [ \cdots \big[ \lambda_{q_1}, \lambda_i ] \cdots ] \theta_{q_1}^* \cdots \theta_{q_n}^*\,.
\end{equation}
\end{theorem}
\begin{proof}
Observe that Corollary \ref{corollary_transfergamma} and the preceeding developments do not apply directly to $\theta_i$, because $[ \theta_j^*, \theta_i ] = \delta_{ij}$ whereas we have required the commutator $[b_j,c] = 0$ to vanish in the above. However, we do have
\[
[ \delta, \theta_i ] = \sum_j [ \lambda_j \theta_j^*, \theta_i ] = \sum_j \lambda_j \delta_{ij} = \lambda_i\,.
\]
And we can adjust the calculations above as follows. Firstly, using Lemma \ref{lemma:woop4} we obtain
\[
\delta^p [\delta, \theta_i] = \sum_{i+j=p} \binom{p}{i} [\delta, \theta_i]^{(i+1)} \delta^j = \sum_{i+j=p} \binom{p}{i} [\delta, \lambda_i]^{(i)} \delta^j\,.
\]
Tracking this through the rest of the calculation, as above, yields \eqref{eq:psik}.
\end{proof}

\begin{remark} The low order terms of \eqref{eq:psik} are
\be
\cat{T}(\theta_i) = \theta_i - \lambda_i - \frac{1}{2} \sum_q [\lambda_q, \lambda_i] \theta^*_q + \cdots
\ee
\end{remark}

It is immediate that

\begin{theorem} The closed odd operators $\cat{T}(\theta_i)$ and $\cat{T}(\theta_i^*) = \theta_i^*$ on the linear factorisation $(K \otimes X, d_K + d)$ form a representation of the Clifford algebra $C_m$.
\end{theorem}

By construction, Proposition \ref{prop:equivalencekoszul} is now promoted to an isomorphism of representations of $C_m$, with $\cat{T}(\theta_i), \cat{T}(\theta_i^*)$ acting on $K \otimes X$ and the $\theta_i, \theta_i^*$ acting on $S_m \otimes_k X$. 


\begin{remark}\label{remark:cancellations_transfer} Suppose for each pair $1 \le i,j \le m$ that up to homotopy
\be
[\lambda_i, \lambda_j] \simeq f_{ij} \cdot 1_X
\ee
for elements $f_{ij} \in R$. Then the formula for $\cat{T}(\theta_i)$ simplifies substantially: for any indices $i,j,k$ we have $[\lambda_i, [\lambda_j, \lambda_k]] \simeq [\lambda_i, f_{jk}] = 0$ so that \eqref{eq:psik} becomes
\be
\cat{T}(\theta_i) \simeq \theta_i - \lambda_i - \frac{1}{2} \sum_q f_{qi} \theta^*_q\,.
\ee
\end{remark}

\subsection{The finite model}\label{section:theequivalence}

We proceed to rediscover $Y \l X$ with its Clifford action as a finite model of $Y \otimes X$. Using this we prove in the next section that $\L$ and $\LG$ are equivalent bicategories. The notation is as in Setup \ref{setupforfusion}. We introduce formal variables $\theta_i$ of odd degree and set
\[
S_m = \bigwedge\left( k \theta_1 \oplus \cdots \oplus k \theta_m \right)\,.
\]
The Koszul complex $K$ of the sequence $t_1,\ldots,t_m$ (where $t_i = \partial_{y_i} V$) is defined by
\begin{equation}\label{defn:koszul2}
K = S_m \otimes_k k[y], \qquad d_K = \sum_{i=1}^m t_i \theta_i^*\,.
\end{equation}
According to Strategy \ref{strategy} the first step is to find a finite model of
\[
( K \otimes_{k[y]} ( Y \otimes X ), d_K ) = ( S_m \otimes_k ( Y \otimes X ), d_K )\,.
\]
\textbf{First step.} A finite model for $(K, d_K)$ is obtained using connections. The connection $\nabla$ of \eqref{eq:connection_nabla} extends canonically to an operator on $\widehat{k[y]} \otimes_{k[t]} \Omega^*_{k[t]/k}$ where $\Omega^* = \bigwedge \Omega^1$ is the exterior algebra. By identifying $d t_i$ with $\theta_i$ we may make the identification
\[
\widehat{K} := K \otimes_{k[y]} \widehat{k[y]} \cong \Omega^*_{k[t]/k} \otimes_{k[t]} \widehat{k[y]}
\]
so that $\nabla$ becomes identified with a degree $-1$ $k$-linear operator on $\widehat{K}$.  It is straightforward to check (see \cite[Lemma 8.7]{dm1102.2957}) that $[d_K, \nabla]$ is invertible on the graded submodule $\widehat{K}_{\le -1}$ of negative degree terms (in the usual $\nZ$-grading on the Koszul complex) and we define
\[
H = [d_K, \nabla]^{-1} \nabla\,.
\]

\begin{proposition} $H$ is a $k$-linear splitting homotopy on the complex $(\widehat{K},d_K)$.
\end{proposition}
\begin{proof}
See \cite[Section 8.1]{dm1102.2957}.
\end{proof}

The associated strong deformation retract of $\mathbb{Z}$-graded $k$-complexes is
\begin{equation}\label{eq:originalHdef}
\xymatrix@C+2pc{
(J_V,0) \ar@<-1ex>[r]_{\sigma} & (\widehat{K}, d_K), \ar@<-1ex>[l]_{\pi}
} \quad H
\end{equation}
where $J_V$ is the Jacobi algebra from \eqref{eq:defnjacobian}, $\pi: \widehat{K} \lto J_V$ is the canonical quasi-isomorphism which vanishes on $K_{\le -1}$ and is in degree zero the quotient map $\widehat{k[y]} \lto J_V$, and $\sigma$ is a $k$-linear embedding of $J_V$ into $\widehat{k[y]}$ uniquely determined by $\nabla$. The splitting homotopy $H$, or equivalently the deformation retract \eqref{eq:originalHdef}, is the desired finite model of $\widehat{K}$. Note that \eqref{eq:originalHdef} is \cite[(10.1)]{dm1102.2957}, modulo the sign difference referred to in Remark \ref{remark:strongdefosigns}.

Recall that the tensor product $Y \,\widehat{\otimes}\, X$ of \eqref{eq:completed_tensor_product} is a finite rank matrix factorisation of $U - W$ over $R = k[x,z] \otimes_k \widehat{k[y]}$. The connection $\nabla$ extends to a standard $k[x,z]$-linear flat connection on $R$. A choice of homogeneous basis for $X$ and $Y$ allows us to write $Y \otimes X \cong k[x,y,z] \otimes_k Q$ for some $\mathbb{Z}_2$-graded free $k$-module $Q$, so
\begin{equation}\label{eq:lambdaiso2}
\widehat{K} \otimes_{\widehat{k[y]}}( Y \,\widehat{\otimes}\, X ) \cong \widehat{K} \otimes_{\widehat{k[y]}} (R \otimes_k Q) \cong \widehat{K} \otimes_k k[x,z] \otimes_k Q
\end{equation}
In this way the $k$-linear splitting homotopy $H$ on $\widehat{K}$ induces a $k[x,z]$-linear operator $H \otimes 1$ on $\widehat{K} \otimes_{\widehat{k[y]}}( Y \,\widehat{\otimes}\, X)$, which we again denote by $H$. It is immediate that

\begin{lemma} $H$ is a $k[x,z]$-linear splitting homotopy on $(\widehat{K} \otimes_{\widehat{k[y]}} ( Y \,\widehat{\otimes}\, X ), d_K)$.
\end{lemma}

The associated strong deformation retract is
\[
\xymatrix@C+2pc{
(Y \l X,0) \ar@<-1ex>[r]_-{\sigma} & (\widehat{K} \otimes_{\widehat{k[y]}} Y \,\widehat{\otimes}\, X, d_K), \ar@<-1ex>[l]_-{\pi}
} \quad H
\]
The map $\pi$ is $R$-linear induced by the map $\pi: \widehat{K} \lto J_V$, and $\sigma = 1 \otimes \sigma$ is defined using map $\sigma$ of \eqref{eq:originalHdef} and the isomorphism \eqref{eq:lambdaiso2}. This is the desired finite model of $K \otimes Y \otimes X$ equipped with only the differential $d_K$. 

\vspace{0.3cm}

\textbf{Second step.} We now view $d_{Y \otimes X}$ as a perturbation. The transference problem asks how to define a splitting homotopy $\phi_\infty$ on the matrix factorisation $(\widehat{K} \otimes_{\widehat{k[y]}} Y \,\widehat{\otimes}\, X , d_K + d_{Y \otimes X})$ in such a way that the underlying graded module is $Y \l X$. Since $H d_{Y \otimes X}$ has finite order, it follows from the perturbation lemma (Theorem \ref{theorem:pertlemma}) with $\tau = d_{Y \otimes X}, \phi = H$ that
\[
\phi_\infty = \sum_{m \ge 0} (-1)^m (H d_{Y \otimes X})^m H
\]
is $k[x,z]$-linear splitting homotopy achieving the desired aim. That is, by \cite[Proposition 7.1]{dm1102.2957} $\phi_\infty$ is a $k[x,z]$-linear splitting homotopy on $\widehat{K} \otimes_{\widehat{k[y]}} Y \,\widehat{\otimes}\, X$ and with
\[
\sigma_\infty = \sum_{m \ge 0} (-1)^m (H d_{Y \otimes X})^m \sigma
\]
there is a $k[x,z]$-linear strong deformation retract of linear factorisations of $U - W$
\begin{equation}\label{eq:finite_model_defo}
\xymatrix@C+4pc{
\big( Y \l X, d_{Y \otimes X} \big) \ar@<-1ex>[r]_-{\sigma_\infty} & \big( \widehat{K} \otimes_{\widehat{k[y]}} Y \,\widehat{\otimes}\, X, d_K + d_{Y \otimes X} \big)\,, \ar@<-1ex>[l]_-{\pi}
} \quad \phi_\infty\,.
\end{equation}

\textbf{Third step.} Since each $t_i$ acts null-homotopically on $Y \otimes X$ there is by Section \ref{section:cliffordactkos} an isomorphism of linear factorisations
\begin{equation}\label{eq:iso_larger_object}
\xymatrix@C+3pc{ \big( \widehat{K} \otimes_{\widehat{k[y]}} Y \otimes X, d_K + d_{Y \otimes X} \big) \ar@<-1ex>[r]_-{ \exp(\delta) } & \big( S_m \otimes_k Y \,\widehat{\otimes}\, X, d_{Y \otimes X} \big)\ar@<-1ex>[l]_-{ \exp(-\delta) } }
\end{equation}
where $\delta = \sum_i \lambda_i \theta_i^*$ for a choice $\lambda_i$ of null-homotopy for the action of $t_i$. From the identity $d^2_X = V - W$ we deduce that $\partial_{y_i}(d_X) d_X + d_X \partial_{y_i}(d_X) = \partial_{y_i} V = t_i$ so that we may, and do, choose $\lambda_i = \partial_{y_i}(d_X)$, see Section \ref{section:partial}.

\vspace{0.3cm}

\textbf{Fourth step.} Combining \eqref{eq:finite_model_defo} and \eqref{eq:iso_larger_object} we have a homotopy equivalence of matrix factorisations over $k[x,z]$
\begin{equation}\label{eq:final_finite_model2}
\xymatrix@C+4pc{
Y \l X \ar@<-1ex>[r]_-{\sigma_\infty} & \widehat{K} \otimes_{\widehat{k[y]}} Y \,\widehat{\otimes}\, X \ar@<-1ex>[l]_-{\pi} \ar@<-1ex>[r]_-{ \exp(\delta) } & S_m \otimes_k ( Y \,\widehat{\otimes}\, X ) \ar@<-1ex>[l]_-{ \exp(-\delta) }
}\,.
\end{equation}
By \cite[Remark 7.7]{dm1102.2957} the canonical map $\epsilon: Y \otimes X \lto Y \,\widehat{\otimes}\, X$ is a homotopy equivalence over $k[x,z]$ (here we use that $k[y] \lto \widehat{k[y]}$ is flat, see Remark \ref{remark:flatness}) so finally we have

\begin{theorem}\label{theorem:htpy_equivalence_main} There is a homotopy equivalence of matrix factorisations over $k[x,z]$
\begin{equation}\label{eq:final_finite_model}
\xymatrix@C+4pc{
Y \l X \ar@<-1ex>[r]_-{\Phi^{-1}} & S_m \otimes_k ( Y \otimes X )\ar@<-1ex>[l]_-{\Phi}
}
\end{equation}
where $\Phi = \pi \exp(-\delta) \epsilon$ and $\Phi^{-1} = \epsilon^{-1} \exp(\delta) \sigma_\infty$.
\end{theorem}

The canonical action of $C_m$ on the spinor representation $S_m$ by wedge product and contraction $\{ \theta_i, \theta_i^* \}_{i=1}^m$ (see Lemma \ref{defn:spinorrep}) induces an action  of $C_m$ on $S_m \otimes_k ( Y \otimes X )$. In Definition \ref{defn:cliffordaction_cut} we gave an action of $C_m$ on the cut $Y \l X$ by operators $\{ \ferm_i, \fermc_i \}_{i=1}^m$, and it remains to show that these two actions are compatible under \eqref{eq:final_finite_model}.

\begin{proposition}\label{prop:clifford_action} The map $\Phi$ is an isomorphism of $C_m$- representations in the homotopy category. That is, there are homotopies for $1 \le i \le m$
\[
\fermc_i \Phi \simeq \Phi \theta_i \,, \qquad \ferm_i \Phi \simeq \Phi \theta_i^*\,.
\]
\end{proposition}
\begin{proof}
We prove that $\fermc_i \simeq \Phi \theta_i \Phi^{-1}$ and $\ferm_i \simeq \Phi \theta_i^* \Phi^{-1}$. With the diagram \eqref{eq:final_finite_model2} in mind, and following the terminology of Section \ref{section:cliffordactkos}, we define the transfers of $\theta_i, \theta_i^*$ from closed operators on $S_m \otimes_k ( Y \,\widehat{\otimes}\, X )$ to operators on the intermediate object $\widehat{K} \otimes_{\widehat{k[y]}} Y \,\widehat{\otimes}\, X$ by
\[
\cat{T}(\theta_i) = \exp(-\delta) \theta_i \exp(\delta)\,, \qquad \cat{T}(\theta_i^*) = \exp(-\delta) \theta_i^* \exp(\delta)\,.
\]
Then we use Lemma \ref{lemma:commutator_of_lambdas}, according to which there are homotopies
\[
[ \lambda_j, \lambda_k ] = [ \partial_{y_j}(d_X), \partial_{y_k}(d_X) ] \simeq \partial_{y_jy_k}(V)\,.
\]
Therefore by Theorem \ref{theorem:psik} and Remark \ref{remark:cancellations_transfer} we have $\cat{T}(\theta_i^*) = \theta_i^*$ and a homotopy
\be\label{eq:comp338h}
\cat{T}(\theta_i) \simeq \theta_i - \lambda_i - \frac{1}{2}\sum_q \partial_{y_q} \partial_{y_i}(V) \theta_q^*\,.
\ee
It remains to compute the further transfer of these operators to $Y \l X$. The key point is proven in Lemma \ref{lemma:transfercreations} below, where we show that there is a homotopy
\be\label{eq:transfercr}
\pi \theta_{q_1}^* \cdots \theta_{q_l}^* \sigma_\infty \simeq \At_{q_1} \cdots \At_{q_l}
\ee
for any indices $q_1,\ldots,q_l$. Hence
\[
\Phi \theta_i^* \Phi^{-1} \simeq \pi \cat{T}(\theta_i^*) \sigma_\infty = \pi \theta_i^* \sigma_\infty \simeq \At_i = \ferm_i\,.
\]
Since $\pi$ projects onto $\theta$-degree zero,
\begin{align*}
\Phi \theta_i \Phi^{-1} &\simeq \pi \cat{T}(\theta_i) \sigma_\infty\\
&\simeq - \pi \lambda_i \sigma_\infty - \frac{1}{2} \sum_q \partial_{y_q} \partial_{y_i}(V) \pi \theta_q^* \sigma_\infty\\
&\simeq - \lambda_i - \frac{1}{2}\sum_q \partial_{y_q} \partial_{y_i}(V) \At_q\\
&= \fermc_i
\end{align*}
which completes the proof.
\end{proof}

\begin{lemma}\label{lemma:transfercreations} In the context of the diagram \eqref{eq:final_finite_model2} there is a $k[x,z]$-linear homotopy
\[
\pi \theta_{q_1}^* \cdots \theta_{q_l}^* \sigma_\infty \simeq \At_{q_1} \cdots \At_{q_l}
\]
for any sequence of indices $1 \le q_1,\ldots,q_l \le m$.
\end{lemma}
\begin{proof}
This is implicit in \cite[Section 10]{dm1102.2957}. The formula $\sigma_\infty = \sum_{m \ge 0} (-1)^m(H d_{Y \otimes X})^m \sigma$ is given there, and with $t_i = \partial_{z_i} V$ the calculation that takes place in \cite[(10.3)]{dm1102.2957} yields
\begin{equation}\label{eq:atiyah_formula_sigma}
\sigma_\infty = \sum_{s \ge 0} \sum_{p_1,\ldots,p_s} (-1)^{\binom{s+1}{2}} \frac{1}{s!} [ d_{Y \otimes X}, \partial_{t_{p_1}}] \cdots [ d_{Y \otimes X}, \partial_{t_{p_s}}] \theta_{p_1} \cdots \theta_{p_s} + (t \text{ terms})\,.
\end{equation}
That is, $\sigma_\infty$ is given modulo the submodule $(t_1,\ldots,t_m) K \otimes Y \otimes X$ by the formula above. Since $\pi$ vanishes on this submodule,
\begin{align*}
\pi \theta_{q_1}^* \cdots \theta_{q_l}^* \sigma_\infty &= \sum_{s \ge 0} \sum_{p_1,\ldots,p_s} (-1)^{\binom{s+1}{2}} \frac{1}{s!} \pi \theta_{q_1}^* \cdots \theta_{q_l}^* \At_{p_1} \cdots \At_{p_s} \theta_{p_1} \cdots \theta_{p_s}\\
&= \sum_{s \ge 0} \sum_{p_1,\ldots,p_s} (-1)^{\binom{s+1}{2} + sl} \frac{1}{s!} \At_{p_1} \cdots \At_{p_s} \pi \theta_{q_1}^* \cdots \theta_{q_l}^* \theta_{p_1} \cdots \theta_{p_s}\,.
\end{align*}
Now $\pi$ is non-vanishing only on terms of $\theta$-degree zero, so the sum restricts to $s = l$ and to $\bold{p}$ a permutation of $\bold{q}$. Since
\[
\theta_{q_1}^* \cdots \theta_{q_l}^* \theta_{q_{\sigma(1)}} \cdots \theta_{q_{\sigma(l)}} = (-1)^{|\sigma|} (-1)^{\binom{l}{2}}
\]
we have, using the fact that the Atiyah classes anti-commute,
\begin{align*}
\pi \theta_{q_1}^* \cdots \theta_{q_l}^* \sigma_\infty &= \sum_{\sigma \in \mathfrak{S}_l} (-1)^{|\sigma|} \frac{1}{l!} \At_{q_{\sigma(1)}} \cdots \At_{q_{\sigma(l)}}\\
&\simeq \At_{q_1} \cdots \At_{q_l}
\end{align*}
which completes the proof.
\end{proof}

Finally, we will need naturality of Theorem \ref{theorem:htpy_equivalence_main}.

\begin{proposition}\label{prop:naturality_main_thm} If $\phi: X \lto X'$ and $\psi: Y \lto Y'$ are morphisms of matrix factorisations the diagram
\begin{equation}\label{eq:final_finite_naturality}
\xymatrix@C+4pc@R+1pc{
Y \l X \ar[d]_-{\psi \l \phi} & S_m \otimes_k ( Y \otimes X )\ar[l]_-{\Phi} \ar[d]^-{\psi \otimes \phi}\\
Y' \l X' & S_m \otimes_k ( Y' \otimes X' )\ar[l]^-{\Phi}
}
\end{equation}
commutes up to $k[x,z]$-linear homotopy.
\end{proposition}
\begin{proof}
Given the formula $\Phi = \pi \exp(-\delta) \epsilon$ it is clearly enough to show that $\psi \l \phi$ commutes up to homotopy with $\delta$, and for this it is enough to show that there is a homotopy
\[
\partial_{y_i}(d_{X'}) \phi \simeq \phi \partial_{y_i}(d_X)
\]
for each $i$. But this is the content of Lemma \ref{lemma:naturalityoflambda}.
\end{proof}

This achieves our aim: we took direct sums of copies of $Y \otimes X$ to form $S_m \otimes_k( Y \otimes X )$ and then used perturbation to find a finite model of this latter object. It is worth noting that, in the language of splitting homotopies, the cut $Y \l X$ is the solution of a fixed point problem formulated in terms of linear operators on the initial data $S_m \otimes_k( Y \otimes X )$. 

\begin{remark} The isomorphism of $C_m$-representations in the homotopy category
\[
\Phi: S_m \otimes_k (Y \otimes X) \lto Y \l X
\]
is straightforward to compute. Firstly, since $\Phi = \pi \exp(-\delta) \varepsilon$ and the operator $\delta = \sum_i \lambda_i \theta_i^*$ annihilates the identity of the exterior algebra $1 \in S_m$, the composite
\be
\xymatrix@C+2pc{
( k \cdot 1 ) \otimes_k ( Y \otimes X ) \ar[r]^-{\operatorname{inc}} & S_m \otimes_k ( Y \otimes X ) \ar[r]^-{\Phi} & Y \l X
}
\ee
is just the quotient map $Y \otimes X \lto Y \l X$, which we write as $\nu \mapsto \bar{\nu}$. Using
\[
\exp(-\delta) = 1 - \sum_i \lambda_i \theta_i^* - \frac{1}{2}\sum_{i,j} \lambda_i \lambda_j \theta_i^* \theta_j^* + \cdots
\]
we have for example
\begin{align*}
\Phi( \theta_1 \theta_2 \otimes \nu ) &= - \frac{1}{2} \sum_{i,j} \pi \lambda_i \lambda_j  \theta_i^* \theta_j^*( \theta_1 \theta_2 \otimes \nu )\\
&= \frac{1}{2}\left( \lambda_1 \lambda_2 -\lambda_2 \lambda_1 \right)( \bar{\nu} )\,.
\end{align*}
\end{remark}

\subsection{The equivalence of $\L$ and $\LG$}\label{section:equivalenceforreal}

The bicategories $\LG$ and $\L$ have the same objects. Their morphism categories are 
\begin{align*}
\L(W,V) &= \Big( \hmf( k[x,y], V - W )^{\omega} \Big)^{\bullet}\,,\\
\LG(W,V) &= \hmf( k[x,y], V - W )^{\oplus}
\end{align*}
with the notation from Section \ref{section:superbicatLG}. These are equivalent supercategories, via the following chain of equivalences
\be
\xymatrix@C-1pc{
& \ar[dl]_-{\iota}\hmf( k[x,y], V - W )^{\omega}\ar[dr]^-{F}\\
\Big( \hmf( k[x,y], V - W )^{\omega} \Big)^{\bullet} &  & \hmf( k[x,y], V - W )^{\oplus}\ar@{.>}[ll]^-{\cat{Z}_{W,V}}
}
\ee
where $\iota$ is the equivalence of a supercategory with its Clifford thickening, $F$ is the functor given earlier in \eqref{eq:omegavsoplus} and we define $\cat{Z}_{W,V} = \iota \circ F^{-1}$ to make the diagram commute. 

The functor $F$ is defined by choosing for each object $(X,e)$ a (possibly infinite rank) matrix factorisation $F(X,e)$ which splits the idempotent $e$. That is, such that there are morphisms in the homotopy category
\[
\xymatrix@C+2pc{
X \ar@<1ex>[r]^-{f} & F(X,e) \ar@<1ex>[l]^-{g}
}
\]
with $f g = 1$ and $gf = e$. The inverse is defined by choosing, for each infinite rank matrix factorisation $T$ which is a summand in the homotopy category of something finite rank, an actual finite rank $X$ and morphisms $f,g$ with $fg = 1_T$, and defining
\[
F^{-1}(T) = (X,gf)\,.
\]
We may arrange these choices so that if $T$ is \emph{already} finite rank then $F^{-1}(T) = (T, 1_T)$.

\begin{definition} A \emph{strong superfunctor} $F: \cat{C} \lto \cat{D}$ between superbicategories without units is a strong superfunctor in the sense of Definition \ref{defn:laxsuperfunctor}, but without the data related to units. We say further that $F$ is an \emph{equivalence} if 
\begin{itemize}
\item for all $a,b \in \cat{C}$, $F_{a,b}: \cat{C}(a,b) \lto \cat{D}(Fa, Fb)$ is an equivalence of categories, and
\item for all $c \in \cat{D}$ there exists $a \in \cat{C}$ with $Fa \cong c$ in the sense that there are $1$-morphisms $f: Fa \lto c$ and $g: c \lto Fa$ with $fg \cong 1$ and $gf \cong 1$.
\end{itemize}
\end{definition}

\begin{theorem} There is an equivalence $\cat{Z}: \LG \lto \L$ of superbicategories without units.
\end{theorem}
\begin{proof}
We define $\mathcal{Z}$ to be the identity on objects and $\mathcal{Z}_{W,V}$ on $1$- and $2$-morphisms. The only remaining content is to produce a $2$-isomorphism
\be\label{eq:superfuncvarphi}
\varphi_{Y,X}: \cat{Z}_{V,U}(Y) \l \cat{Z}_{W,V}(X) \lto \cat{Z}_{W,U}( Y \otimes X )
\ee
in the usual situation of Setup \ref{setupforfusion}, which is natural in both $Y$ and $X$. Once we have done this the rest follows by functoriality.

Since $Y,X$ in this case are already finite rank $\cat{Z}(Y) = Y$ and $\cat{Z}(X) = X$.  The calculation of $\cat{Z}( Y \otimes X )$ is more interesting, because $Y \otimes X$ is infinite rank over $k[x,z]$. However we know from Proposition \ref{prop:clifford_action} that there is an isomorphism in the homotopy category of (infinite rank) matrix factorisations of $U - W$ over $k[x,z]$
\begin{equation}\label{eq:final_finite_modelagain}
\xymatrix@C+4pc{
Y \l X \ar@<-1ex>[r]_-{\Phi^{-1}} & S_m \otimes_k ( Y \otimes X )\ar@<-1ex>[l]_-{\Phi}
}
\end{equation}
where $Y \l X$ is finite rank. Recall from Definition \ref{defn:idempotent_e} that the idempotent
\[
e_m = \ferm_1 \cdots \ferm_m \fermc_m \cdots \fermc_1 \in C_m
\]
represents projection onto $k \cdot 1$ in $S_m$. Since \eqref{eq:final_finite_modelagain} is an isomorphism of $C_m$-representations the idempotent $e_m$ acting on $Y \l X$ must split to the same thing as the idempotent $e_m$ acting on the right hand side, which obviously splits to $Y \otimes X$. Put differently, there is a diagram in the homotopy category
\[
\xymatrix@C+2pc{
Y \l X \ar@<1ex>[r]^-{f} & Y \otimes X \ar@<1ex>[l]^-{g}
}
\]
where $g$ is the quotient map, and $f g = 1_{Y \otimes X}$ while $g f = e_m$.\footnote{$f$ is $\Phi^{-1}$ followed by projection onto $\theta$-degree zero.} Then we may take as our definition
\[
\cat{Z}( Y \otimes X ) = (Y \l X, e_m)\,.
\]
So the problem is to define a $2$-isomorphism \eqref{eq:superfuncvarphi}, that is,
\[
\varphi_{Y,X}: (Y \l X, 1, \{ \ferm_i, \fermc_i \}_{i=1}^m) \lto (Y \l X, e_m)
\]
in the Clifford thickening of the idempotent completion of $\hmf(k[x,z],U - W)$, where the left hand side is equipped with the idempotent $1$ and the action $\{ \ferm_i, \fermc_i \}_{i=1}^m$ of $C_m$ while the right hand side has no Clifford action and idempotent $e_m$.

But by Lemma \ref{lemma:whackamole} the identity map $\varphi_{Y,X} = 1_{Y \l X}$ is such an isomorphism, because the object $(Y \l X, e_m)$ of the idempotent completion by definition splits the idempotent $e_m$ on $Y \l X$. Naturality of \eqref{eq:superfuncvarphi} in $Y,X$ is a consequence of the naturality of \eqref{eq:final_finite_modelagain}, see Proposition \ref{prop:naturality_main_thm}.
\end{proof}

\begin{remark} Obviously the $1$-morphisms $\Delta_W: W \lto W$ which are the units in $\LG$ are, by virtue of the equivalence $\mathcal{Z}$, also units in $\L$, so that this is an honest bicategory and $\mathcal{Z}$ an equivalence of bicategories. The reason we do not begin with units in $\L$ is that the unitors are no simpler in $\L$ than $\LG$, so there is no real point to introducing them.
\end{remark}

\subsection{Example: the Hom complex}\label{example:computing_homs}

Using the bicategory $\L$ we can study the complex $\Hom_R(Y,X)$ for two finite rank matrix factorisations $X,Y$ of a potential $V \in R = k[y_1,\ldots,y_m]$. The pair $(k,0)$ is an object in $\LG$ and we may view $X$ and the dual matrix factorisation $Y^{\vee}$ of $-V$ as a pair of $1$-morphisms
\begin{equation}\label{eq:compute_hom_composition}
\xymatrix@C+2pc{
0 \ar[r]^-{X} & V \ar[r]^-{Y^{\vee}} & 0
}\,.
\end{equation}
The composite is a $1$-morphism $0 \lto 0$, that is, $\mathbb{Z}_2$-graded complex of $k$-modules. This is the Hom complex of $R$-linear homogeneous maps $Y \lto X$,
\[
Y^{\vee} \otimes_R X \cong \Hom_R(Y, X)
\]
with the differential $\phi \mapsto d_X \circ \phi - (-1)^{|\phi|} \phi \circ d_Y$. The cut is
\[
Y^{\vee} \l X = Y^{\vee} \otimes_R J_V \otimes_R X \cong \Hom_{J_V}( \bar{Y}, \bar{X} )
\]
where $\bar{X} = X \otimes_R J_V, \bar{Y} = Y \otimes_R J_V$. There is by Theorem \ref{theorem:htpy_equivalence_main} a $C_m$-linear isomorphism in the homotopy category of $\mathbb{Z}_2$-graded complexes of free $k$-modules
\be
\Phi: S_m \otimes_k \Hom_R(Y,X) \lto Y^{\vee} \l X\,.
\ee

\begin{example} Suppose $V = y^N$ in $k[y]$ for $N \ge 2$. Take matrix factorisations $X,Y$ with the same underlying free $R = k[y]$-module $X = Y = R \oplus R \theta$ and differentials
\[
d_Y = y^i \theta + y^{N-i} \theta^* = \begin{pmatrix} 0 & y^{N-i} \\ y^i & 0 \end{pmatrix}, \qquad d_X = y^j \theta + y^{N-j} \theta^* = \begin{pmatrix} 0 & y^{N-j}\\ y^j & 0 \end{pmatrix}
\]
where $1 \le i,j \le N/2$. Using the basis $\Hom_R(Y,X) = R \theta \theta^* \oplus R \theta^* \theta \oplus R\theta \oplus R \theta^*$,
\begin{equation}\label{eq:differential_hom}
d_{\Hom} = \begin{pmatrix} 0 & 0 & y^{N-i} & y^j \\ 
0 & 0 & y^{N-j} & y^i \\
-y^i & y^j & 0 & 0 \\
y^{N-j} & -y^{N-i} & 0 & 0 \end{pmatrix}\,.
\end{equation}
Let us now compute the cut $Y^{\vee} \l X$, according to Definition \ref{defn:cliffordaction_cut}. With
\be\label{eq:examplecut_basis2}
J_V = k[y]/y^{N-1} = k \cdot 1 \oplus k \cdot y \oplus \cdots \oplus k \cdot y^{N-2}
\ee
the matrix of multiplication by $y$ on $J_V$ is $[y] = \begin{pmatrix} 0 & 0\\ I_{N-2} & 0 \end{pmatrix}$. The differential on
\be\label{eq:examplecut_basis}
Y^{\vee} \l X = J_V \theta \theta^* \oplus J_V \theta^* \theta \oplus J_V\theta \oplus J_V \theta^*
\ee
is then obtained from \eqref{eq:differential_hom} by replacing every $y$ by the matrix $[y]$.

The Clifford algebra $C_1$ acts on $Y^{\vee} \l X$ by generators $\ferm, \fermc$. To determine their matrices in the basis determined by \eqref{eq:examplecut_basis2} and \eqref{eq:examplecut_basis}, note that with $t = \partial_y(V) = N y^{N-1}$, $k[y]$ is a free $k[t]$-module with basis $1, \ldots, y^{N-2}$. Given a polynomial $f(y)$ we may uniquely write it as $f(y) = \sum_{l = 0}^{N-2} f_l \cdot y^l$ with $f_l \in k[t]$. There is a connection $\nabla: k[y] \lto k[y] \otimes_{k[t]} \Omega^1_{k[t]/k}$ whose associated $k$-linear operator $\partial_t$ is
\[
\partial_t: k[y] \lto k[y], \qquad \partial_t(f) = \sum_{l = 0}^{N-2} \frac{\partial}{\partial t}( f_l ) y^l
\]
where $\frac{\partial}{\partial t}$ is the usual operator on $k[t]$. For $q < N - 1$ we have $\partial_t(y^q) = 0$ and $\partial_t(y^{N-1+q}) = \frac{1}{N} y^q$, so the $k$-linear operator
\[
R_a: J_V \lto J_V \,, \qquad y^q \mapsto \partial_t( y^{a+q} )
\]
has the following block form for $0 \le a \le N - 2$
\[
[R_a] = \frac{1}{N} \begin{pmatrix} 0 & I_{a} \\ 0 & 0 \end{pmatrix}\,.
\]
The Atiyah class is the operator $\At = [d_{\Hom}, \partial_t] = d_{\Hom} \partial_t - \partial_t d_{\Hom}$. For example,
\begin{align*}
\At( y^q \theta \theta^* ) &= - \partial_t d_{\Hom}( y^q \theta \theta^* ) = \partial_t( y^{i+q} ) \theta - \partial_t( y^{N-j+q} ) \theta^*\,.
\end{align*}
The matrix of the Atiyah class as an operator $\ferm$ on $Y^{\vee} \l X$ is therefore
\begin{equation}\label{eq:aityah_class_hom_example}
\ferm = \begin{pmatrix} 0 & 0 & -[R_{N-i}] & -[R_{j}] \\ 
0 & 0 & -[R_{N-j}] & -[R_{i}] \\
[R_i] & -[R_j] & 0 & 0 \\
-[R_{N-j}] & [R_{N-i}] & 0 & 0 \end{pmatrix}\,.
\end{equation}
Whereas the matrix of $\partial_{y}(d_X) = j y^{j-1} \theta + (N-j) y^{N-j-1} \theta^*$ on $Y^{\vee} \l X$ is
\[
\partial_y(d_X) = \begin{pmatrix} 0 & 0 & 0 & j [y^{j-1}] \\
0 & 0 & (N-j) [y^{N-j-1}] & 0\\
0 & j [y^{j-1}] & 0 & 0\\
(N-j) [y^{N-j-1}] & 0 & 0 & 0\end{pmatrix}\,.
\]
By definition
\[
\fermc = - \partial_y(d_X) - \frac{1}{2} N (N-1) y^{N-2} \ferm
\]
which completes the description of the cut $( Y^{\vee} \l X, \{ \ferm, \fermc \} )$. 
\end{example}


\appendix

\section{Tensor products in supercategories}\label{section:tensorproduct_supcat}

Let $\cat{C}$ be an idempotent complete supercategory.

\begin{definition} Let $V$ be a $\mathbb{Z}_2$-graded $k$-module and $X$ an object of $\cat{C}$. An object of $\cat{C}$ representing the functor $\Hom^0_k( V, \cat{C}^*(X, -))$ is denoted $V \otimes_k X$ if it exists. That is to say, the tensor product consists of an object $V \otimes_k X$ and a natural isomorphism
\[
\rho_Y: \cat{C}( V \otimes_k X, Y ) \lto \Hom^0_k( V, \cat{C}^*(X,Y) )\,.
\]
Such a pair is unique up to unique isomorphism, if it exists.
\end{definition}

\begin{remark} 
From $\rho$ we also obtain an isomorphism of $\mathbb{Z}_2$-graded $k$-modules
\[
\cat{C}^*( V \otimes_k X, Y ) \cong \Hom_k^*( V, \cat{C}^*(X,Y) )\,.
\]
The tensor product $V \otimes_k X$ is made functorial in $V$ and $X$ in such a way that $\rho$ is natural in all its variables.
\end{remark}

\begin{example} If $V = k^{\oplus n} \oplus k[1]^{\oplus m}$ is finite and free then
\[
V \otimes_k X = X^{\oplus n} \oplus \Psi X ^{\oplus m}
\]
is a representing object, with the obvious isomorphism $\rho$.
\end{example}

\begin{lemma} If $\cat{C}$ is idempotent complete and $V$ is a finitely generated projective $\mathbb{Z}_2$-graded $k$-module, then $V \otimes_k X$ exists for any object $X$.
\end{lemma}
\begin{proof}
Let $V'$ be a finite free $k$-module with $f: V' \lto V, g: V \lto V'$ satisfying $fg = 1$. This determines an idempotent $1 \otimes e$ on $V' \otimes_k X$, which splits by hypothesis in $\cat{C}$. The object splitting this idempotent has the right property to be $V \otimes_k X$.
\end{proof}

More generally, let $A$ be a $\mathbb{Z}_2$-graded $k$-algebra.

\begin{definition} Given a right $\mathbb{Z}_2$-graded $A$-module $V$ and an object $X$ of $\cat{C}$, we denote by $V \otimes_A X$ the object representing the functor $\Hom^0_A( V, \cat{C}^*(X,-))$ if it exists. That is, the tensor product is an object $V \otimes_A X$ together with a natural isomorphism
\begin{equation}\label{eq:isodefatensor}
\rho_Y: \cat{C}(V \otimes_A X, Y) \lto \Hom^0_A( V, \cat{C}^*(X,Y))\,.
\end{equation}
As before, the tensor product is functorial in both $V$ and $X$.
\end{definition}

Henceforth we concentrate our attention on $\mathbb{Z}_2$-graded $k$-algebras $A$ which are \emph{Morita trivial} in the sense that they are isomorphic to an algebra of the form $\End_k(P)$ for a finite rank free $\mathbb{Z}_2$-graded $k$-module $P$.

\begin{lemma} If $P$ is a finite rank free $\nZ_2$-graded $k$-module and $A \cong \End_k(P)$, then for any $A$-module $X$ in $\cat{C}$ there is an object $\widetilde{X}$ and an isomorphism of $A$-modules $X \cong P \otimes_k \widetilde{X}$.
\end{lemma}
\begin{proof}
The natural idempotents in $A$ act as idempotents on $X$, which split.
\end{proof}

\begin{lemma} If $A$ is Morita trivial and $V$ is a $\mathbb{Z}_2$-graded right $A$-module which is finitely generated and projective as a $k$-module, then $V \otimes_A X$ exists for any $A$-module $X$.
\end{lemma}
\begin{proof}
Write $A \cong \End_k(P)$ and $P^* = \Hom_k(P,k)$ so that $V \cong \widetilde{V} \otimes_k P^*$ as right $A$-modules for some $\mathbb{Z}_2$-graded $k$-module $\widetilde{V}$. Since $V$ is finitely generated projective over $k$, so is $\widetilde{V}$. By the previous lemma, $X \cong P \otimes_k \widetilde{X}$ for some object $\widetilde{X}$, so we might expect that
\[
V \otimes_A X \cong (\widetilde{V} \otimes_k P^*) \otimes_A (P \otimes_k \widetilde{X}) \cong \widetilde{V} \otimes_k \widetilde{X}
\]
Working backwards: $\widetilde{V} \otimes_k \widetilde{X}$ exists and has the right universal property since
\begin{align*}
\cat{C}(\widetilde{V} \otimes_k \widetilde{X}, Y) &\cong \Hom_k( V \otimes_A P, \cat{C}^*(\widetilde{X}, Y))\\
&\cong \Hom_A( V, P^* \otimes_k \cat{C}^*(\widetilde{X}, Y))\\
&\cong \Hom_A( V, \cat{C}^*( X, Y ) )
\end{align*}
as claimed.
\end{proof}

From now on $A,B,C$ are Morita trivial $k$-algebras. Recall from Definition \ref{defn:algebra_modules} that $\cat{C}_A$ denotes the category of left $A$-modules in $\cat{C}$.

\begin{definition}\label{defn:psimodule} Given a left $A$-module $X$ in $\cat{C}$, the object $\Psi X$ is a left $A$-module, where $a \in A_0$ acts by $\Psi(a)$ and $a \in A_1$ acts by
\[
\xymatrix{
\Psi (X) \ar[r]^-{\Psi(a)} & \Psi \Psi (X) \ar[r]^-{-1} & \Psi \Psi (X)
}\,.
\]
This determines a functor $\Psi: \cat{C}_A \lto \cat{C}_A$.
\end{definition}

\begin{lemma} $(\cat{C}_A, \Psi)$ is a supercategory.
\end{lemma}
\begin{proof}
The natural isomorphism $\xi_X: \Psi \Psi (X) \lto X$ is given by the identity in $\cat{C}$.
\end{proof}

Let $V$ be a $B$-$A$-bimodule which is finitely generated and projective over $k$. If $X$ is a left $A$-module in $\cat{C}$ then $V \otimes_A X$ is made into a left $B$-module in $\cat{C}$ uniquely such that \eqref{eq:isodefatensor} is a natural isomorphism of $B$-modules. Hence the tensor product gives a functor
\[
\Phi_V = V \otimes_A (-): \cat{C}_A \lto \cat{C}_B\,.
\]
Given a $B$-$A$-bimodule $V$ and an object $X$ for which $V \otimes_A X$ exists, there are natural isomorphisms of $B$-modules
\[
\Psi V \otimes_A X \cong \Psi(V \otimes_A X) \cong V \otimes_A \Psi X
\]
induced from the isomorphisms
\[
\Hom_A( \Psi V, \cat{C}^*(X, -)) \cong \Hom_A( V, \Psi \cat{C}^*(X, -) ) \cong \Hom_A( V, \cat{C}^*(\Psi X, -))\,.
\]
It follows that $\Phi_V$ is a superfunctor.

\begin{lemma}\label{lemma:assoctensor} Let $W$ be a $C$-$B$-bimodule and $V$ a $B$-$A$-bimodule in $\cat{C}$. Then for any left $A$-module $X$ in $\cat{C}$ there is a natural isomorphism of $C$-modules
\[
(W \otimes_B V) \otimes_A X \cong W \otimes_B ( V \otimes_A X )\,.
\]
\end{lemma}
\begin{proof}
This follows from the isomorphism
\begin{align*}
\Hom_A(W, \cat{C}^*(V \otimes_A X, -)) &\cong \Hom_A( W, \Hom_A^*(V, \cat{C}^*(X, -)))\\
&\cong \Hom_A(W \otimes_A V, \cat{C}^*(X, -))\,.
\end{align*}
\end{proof}

Now we specialise to the Clifford algebras $C_n$ and their modules.

\begin{lemma}\label{lemma:simplehom} Let $X$ be a left $C_n$-module in $\cat{C}$. There is a $k$-linear isomorphism
\[
\Hom_{C_n}^0( S_n, X ) \cong \big\{ x \in X^0 \l \ferm_i \cdot x = 0 \text{ for all } 1 \le i \le n \big\}
\]
defined by $f \mapsto f(1)$. If $X$ is a right $C_n$-module there is a $k$-linear isomorphism
\[
\Hom^0_{C_n}( (S_n)^*, X ) \cong \big\{ x \in X^0 \l x \cdot \fermc_i = 0 \text{ for all } 1 \le i \le n \big\}
\]
defined by $f \mapsto f(1^*)$.
\end{lemma}

\begin{lemma}\label{lemma:morphism_out_1} For an $C_n$-module $Y$ there is a $k$-linear isomorphism
\[
\cat{C}_{C_n}( S_n \otimes_k X, Y ) \cong \{ \delta \in \cat{C}(X,Y) \l \ferm_i \circ \delta = 0 \text{ for all } 1 \le i \le n \}
\]
defined by sending $\alpha: S_n \otimes_k X \lto Y$ to the composite
\[
\xymatrix@C+2pc{
X \cong k \otimes_k X \ar[r]^-{\iota \otimes 1} & S_n \otimes_k X \ar[r]^-{\alpha} & Y
}
\]
where $\iota: k \lto S_n$ is defined by $\iota(1) = 1$.
\end{lemma}
\begin{proof}
This follows from Lemma \ref{lemma:simplehom} and the isomorphism
\[
\cat{C}_{C_n}(S_n \otimes_k X, Y ) \cong \Hom_{C_n}( S_n, \cat{C}^*(X,Y) )\,.
\]
\end{proof}

More generally

\begin{lemma}\label{lemma:morphisms_two_forms} Let $X \in C_{C_p}$ and $Y \in \cat{C}_{C_q}$. There is a $k$-linear isomorphism
\begin{align*}
\cat{C}_{C_q}( S_{q,p} \otimes_{C_p} X, Y ) \cong \{ \delta: X \lto Y \l \delta \circ \fermc_i &= 0 \text{ for } 1 \le i \le p \text{ and }\\ \ferm_i \circ \delta &= 0 \text{ for } 1 \le i \le q \}
\end{align*}
defined by sending $\alpha: S_{q,p} \otimes_{C_p} X \lto Y$ to the composite
\[
\xymatrix@C+2pc{
X \cong C_p \otimes_{C_p} X \ar[r]^-{\iota \otimes 1} & S_{q,p} \otimes_{C_p} X \ar[r]^-{\alpha} & Y
}
\]
where $\iota: C_p \lto S_{q,p}$ is defined by $\iota(1) = 1 \otimes 1^*$.
\end{lemma}
\begin{proof}
Using the previous lemma and Lemma \ref{lemma:simplehom},
\begin{align*}
\cat{C}_{C_q}( S_{q,p} \otimes_{C_p} X, Y ) &\cong \Hom^0_{C_p-C_q-\operatorname{bimod}}( S_{q} \otimes_k S_p^* , \cat{C}^*(X,Y))\\
&\cong \Hom^0_{C_q}( S_q, \Hom^*_{C_p}( S_p^*, \cat{C}^*(X,Y) ))\\
&\cong \{ \delta \in \Hom^0_{C_q}( S_p^*, \cat{C}^*(X,Y) ) \l \ferm_i \circ \delta = 0 \text{ for } 1 \le i \le q \}\\
&\cong \{ \delta: X \lto Y \l \delta \circ \fermc_i = 0 \text{ for } 1 \le i \le p \text{ and } \ferm_i \circ \delta = 0 \text{ for } 1 \le i \le q \}
\end{align*}
as claimed.
\end{proof}

\section{Constructing superbicategories}\label{section:constructing_superbicategories}

This appendix collects the data needed for constructing a superbicategory $\cat{B}$ from a collection of supercategories $\cat{B}(a,b)$. Given a bicategory $\cat{B}$ and for each pair of objects $a,b$ the structure of a supercategory on $\cat{B}(a,b)$, we denote composition by
\[
T: \cat{B}(b, c) \otimes_k \cat{B}(a,b) \lto \cat{B}(a,c)\,.
\]
Since our main example is constructing the superbicategory associated to the cut operation, we write $T(Y,X)$ as $Y \l X$. We denote the associator by $\alpha$, the units in $\cat{B}$ by $\Delta_a: a \lto a$, and the unitors for $X: a \lto b$ by $\lambda_X: \Delta_b \l X \lto X$ and $\rho_X: X \l \Delta_a \lto X$.

Suppose that for all objects $a,b,c$ we are given natural isomorphisms
\begin{align*}
T \circ ( \Psi \otimes 1 ) \lto \Psi \circ T, \qquad T \circ ( 1 \otimes \Psi ) \lto \Psi \circ T
\end{align*}
both of which will be denoted $\tau$. We define for an object $a$ the $1$-morphism $\Psi_a = \Psi( \Delta_a )$ and the $2$-isomorphism $\xi_a$ to be the following composite (using the structure maps of $\cat{B}$ and isomorphisms $\tau$)
\[
\xi_a: \Psi_a \l \Psi_a = \Psi( \Delta_a ) \l \Psi( \Delta_a ) \cong \Psi( \Delta_a \l \Psi(\Delta_a ) ) \cong \Psi^2( \Delta_a \l \Delta_a ) \cong \Delta_a \l \Delta_a \cong \Delta_a\,.
\]
Similarly given a $1$-morphism $X: a \lto b$ we have a $2$-isomorphism $\gamma_X$
\[
\gamma_X: X \l \Psi_a = X \l \Psi( \Delta_a ) \cong \Psi( X \l \Delta_a ) \cong \Psi(X) \cong \Psi( \Delta_b \l X ) \cong \Psi( \Lambda_b ) \l X = \Psi_b \l X\,.
\]
Given a $1$-morphism $X: a \lto b$ in $\cat{B}$ there are diagrams
\begin{gather}
\xymatrix@R+1pc@C+1pc{
\Delta_b \l \Psi (X) \ar[dr]_-{\lambda_{\Psi a}} \ar[rr]^-{\tau} & & \Psi( \Delta_b \l X ) \ar[dl]^-{\Psi( \lambda )}\\
& \Psi(X)
} \label{eq:psiunitor1}\\
\xymatrix@R+1pc@C+1pc{
\Psi(X) \l \Delta_a \ar[rr]^-{\tau}\ar[dr]_-{\rho} & & \Psi( X \l \Delta_a ) \ar[dl]^-{\Psi(\rho)}\\
& \Psi(X)
} \label{eq:psiunitor2}
\end{gather}
whose commutativity expresses compatibility of the functors $\Psi$ with the units, while for a composable triple $X,Y,Z$ the commutativity of the diagrams
\begin{gather}
\xymatrix{
(Z \l Y) \l \Psi(X) \ar[dd]_-{\tau} \ar[rr]^-{\alpha} & & Z \l ( Y \l \Psi(X)) \ar[d]^-{\tau}\\
 & & Z \l \Psi( Y \l X ) \ar[d]^-{\tau}\\
\Psi( (Z \l Y) \l X ) \ar[rr]_-{\alpha} & & \Psi( Z \l ( Y \l X ) )
} \label{eq:psithirdlastdia}\\
\xymatrix{
(Z \l \Psi(Y)) \l X \ar[d]_-{\tau} \ar[rr]^-{\alpha} & & Z \l ( \Psi(Y) \l X) \ar[d]^-{\tau}\\
\Psi( Z \l Y ) \l X \ar[d]_-{\tau} & & Z \l \Psi( Y \l X ) \ar[d]^-{\tau}\\
\Psi( (Z \l Y) \l X ) \ar[rr]_-{\alpha} & & \Psi( Z \l ( Y \l X ) )
} \label{eq:psimiddledia}\\
\xymatrix{
(\Psi(Z) \l Y) \l X \ar[d]_-{\tau} \ar[rr]^-{\alpha} & & \Psi(Z) \l ( Y \l X ) \ar[dd]^-{\tau}\\
\Psi(Z \l Y) \l X \ar[d]_-{\tau} & & \\
\Psi( (Z \l Y) \l X ) \ar[rr]_-{\alpha} & & \Psi( Z \l ( Y \l X ) )
}\label{eq:psilastdia}
\end{gather}
expresses compatibility of $\Psi$ with the associator of $\cat{B}$. Finally, commutativity of
\begin{equation}\label{eq:psicomp1}
\xymatrix@R+1pc@C+1pc{
T( 1 \otimes \Psi^2 ) \ar[r]^-{\tau}\ar[d]_-{T(1 \otimes \xi)} & \Psi T( 1 \otimes \Psi ) \ar[r]^-{\tau} & \Psi^2 T \ar[d]^-{\xi * 1_T}\\
T \ar[rr]_-{1_T} & & T
}
\end{equation}
\begin{equation}\label{eq:psicomp2}
\xymatrix@R+1pc@C+1pc{
T( \Psi^2 \otimes 1 ) \ar[r]^-{\tau} \ar[d]_-{T (\xi \otimes 1)}  & \Psi T( \Psi \otimes 1 ) \ar[r]^-{\tau} & \Psi^2 T \ar[d]^-{\xi * 1_T}\\
T \ar[rr]_-{1_T} & & T
}
\end{equation}
\begin{equation}\label{eq:psicomp3}
\xymatrix@R+1pc@C+1pc{
T( \Psi \otimes \Psi ) \ar[r]^-{\tau}\ar[d]_-{-\tau} & \Psi T( 1 \otimes \Psi ) \ar[d]^-{\tau}\\
\Psi T( \Psi \otimes 1 ) \ar[r]_-{\tau} & \Psi^2 T
}
\end{equation}
expresses that $T$ is a superfunctor.

\begin{lemma}\label{lemma:constructingsuper}
Given a bicategory $\cat{B}$ with the structure of a supercategory on $\cat{B}(a,b)$ for all pairs $a,b$, and natural isomorphisms $\tau$, suppose that \eqref{eq:psiunitor1}-\eqref{eq:psicomp3} commute. Then with $\Psi, \xi$ and $\gamma$ defined as above, $\cat{B}$ is a superbicategory.
\end{lemma}
\begin{proof}
The proof is an easy but somewhat lengthy exercise, which we omit.
\end{proof}

\bibliographystyle{amsalpha}
\providecommand{\bysame}{\leavevmode\hbox to3em{\hrulefill}\thinspace}
\providecommand{\href}[2]{#2}

\end{document}